\documentclass{amsart}
\usepackage{latexsym}
\usepackage{graphicx,subfigure}
\usepackage{array}
\usepackage{indentfirst}
\usepackage{fancyhdr}
\usepackage{epsf}
\usepackage{amsmath,bm}
\usepackage{amsthm}
\usepackage{booktabs}
\usepackage{amsfonts}
\usepackage{epsf}
\usepackage{multicol}
\usepackage[square, comma, sort&compress, numbers]{natbib}
\usepackage{fancyhdr}
\usepackage{pxfonts}
\usepackage{caption}
\usepackage{epstopdf}
\usepackage{algorithm}
\usepackage{algpseudocode}
\usepackage{exscale}
\usepackage{relsize}
\usepackage{setspace}
\usepackage{geometry}
\usepackage{tikz}
\usetikzlibrary{arrows,shapes,chains,positioning}
\usetikzlibrary{decorations.markings}
\tikzstyle arrowstyle=[scale=1]
\tikzstyle directed=[postaction={decorate,decoration={markings,mark=at position .65 with {\arrow[arrowstyle]{stealth}}}}]
\tikzstyle reverse directed=[postaction={decorate,decoration={markings,mark=at position .65 with {\arrowreversed[arrowstyle]{stealth};}}}]

\newtheorem{Def}{Definition}
\newtheorem{Th}{Theorem}[section]

\newtheorem{remark}{Remark}[section]

\newtheorem{theorem}{Theorem}[section]
\newtheorem{lemma}[theorem]{Lemma}

\theoremstyle{definition}
\usepackage[titletoc]{appendix}
\numberwithin{equation}{section}
\newcommand{\bc}{\begin{center}}
\newcommand{\ec}{\end{center}}
\newcommand{\be}{\begin{eqnarray}}
\newcommand{\ee}{\end{eqnarray}}

\newcommand{\ben}{\begin{eqnarray*}}
\newcommand{\een}{\end{eqnarray*}}

\newcommand{\Om}{{\rm\Omega}}

\newcommand{\dx}{\,dx}

\newcommand{\ds}{\,ds}

\newcommand{\Rmnum}[1]{\expandafter\@slowromancap\romannumeral #1@}
\newcommand{\Rmath}{\mathbb{R}}

\newcommand{\Tma}{\mathcal{T}_h}

\newcommand{\VM}{V_{\rm M}}
\newcommand{\Mu}{u_{\rm M}}
\newcommand{\Mlam}{\lambda_{\rm M}}
\newcommand{\SHHJ}{\Sigma(\Tma)}
\newcommand{\UHHJ}{U(\Tma)}
\newcommand{\PiHHJ}{\Pi_{\rm HHJ}}
\newcommand{\sHHJ}{\sigma_{\rm HHJ}}
\newcommand{\uHHJ}{u_{\rm HHJ}}
\newcommand{\sHHJc}{\sigma_{\rm HHJ}^{\lambda u}}
\newcommand{\uHHJc}{u_{\rm HHJ}^{\lambda u}}
\newcommand{\Mum}{\widetilde{u}_{\rm M}}
\newcommand{\Mumc}{\widetilde{u}_{\rm M}^{\lambda u}}
\newcommand{\Muc}{u_{\rm M}^{\lambda u}}
\newcommand{\sHHJf}{\sigma_{\rm HHJ}^f}
\newcommand{\uHHJf}{u_{\rm HHJ}^f}
\newcommand{\phiH}{\phi_{\rm HHJ}} 
\newcommand{\phiM}{\phi_{\rm M}} 
 
\newcommand{\Pit}{\Pi^3}
\newcommand{\Pif}{\Pi^4}
\newcommand{\Pik}{\Pi^l}

\newcommand{\PiM}{\Pi_{\rm M}}

\newcommand{\HHJcro}{\gamma_{\rm HHJ}}

\newcommand{\cE}{\mathcal{E}}

\newcommand{\cO}{\mathcal{O}}

\newcommand{\cT}{\mathcal{T}}
\newcommand{\R}{\mathbb{R}}
\def\S{\mathcal{S}}

\newcommand{\bp}{\boldsymbol{p}}
\newcommand{\bM}{\boldsymbol{m}}

\newcommand{\bt}{\boldsymbol{t}}
\newcommand{\bn}{\boldsymbol{n}}
\newcommand{\bx}{\boldsymbol{x}}

\renewcommand{\arraystretch}{1.5}

\newcommand{\yemeifont}{\fontsize{9pt}{\baselineskip}\selectfont}
\begin{document}
\title{
New Fourth Order Postprocessing Techniques for Plate Bending Eigenvalues by Morley Element
}


\author{Limin Ma}
\address{Department of Mathematics, Pennsylvania State University, University Park, PA,
16802, USA. maliminpku@gmail.com}

\author {Shudan Tian}
\address{LMAM and School of Mathematical Sciences, Peking University,
  Beijing 100871, P. R. China.  tianshudan@pku.edu.cn}

\maketitle

\begin{abstract}
In this paper, we propose and analyze the extrapolation method and asymptotically exact a posterior error estimate for eigenvalues of the Morley element. We establish an asymptotic expansion of eigenvalues, and prove an optimal result for this expansion and the corresponding extrapolation method.
We also design an asymptotically exact a posterior error estimate and  propose new approximate eigenvalues with higher accuracy by utilizing this a posteriori error estimate. Finally, several numerical experiments  are considered to confirm the theoretical results and compare the performance of the proposed methods.

  \vskip 15pt

\noindent{\bf Keywords. }{eigenvalue problem, Morley element, extrapolation method,  asymptotically exact a posterior error estimates}

 \vskip 15pt

\noindent{\bf AMS subject classifications.}
    { 65N30}

\end{abstract}

\section{Introduction}
The biharmonic eigenvalue problem originates from the plate theory of elasticity, and also occurs in many physical areas, say the inverse scattering theory. In the Kirchhoff-Love plate model, the biharmonic eigenvalue problem describes the vibration and buckling of an elastic plate subject to some certain boundary condition. The Morley element method is one of the most popular methods for this problem in applied mechanics and engineering, and is widely studied in literature. 
The application of the Morley element in plate problems can be found in \cite{brenner2013morley,ciarlet1974conforming,li2014new,morley1968triangular} and the references therein. Some a posteriori error estimates and adaptive algorithms were established in \cite{da2007posteriori,hu2009new,hu2012convergence}.
For eigenvalue problems by the Morley element, the a priori error estimates were analyzed in \cite{gallistl2015morley,rannacher1979nonconforming} and an a posteriori error estimate was analyzed in \cite{shen2015posteriori}. Guaranteed lower bounds for eigenvalues of the biharmonic equation was proposed and analyzed in \cite{carstensen2014guaranteed,hu2016guaranteed}.

The extrapolation method is an efficient approach to improve the accuracy of approximations of many problems. The key of the efficiency of extrapolation algorithm is an asymptotic expansion of the error. 
The classical analysis of asymptotic expansions is usually based on a superclose property of the canonical interpolation of the element under consideration, 
see \cite{Blum1990Finite,Ding1990quadrature,Lin2008New,Lin2009New,Lin2007Finite,Lin1984Asymptotic,lin1999high,Lin2009Asymptotic,lin2010new,Lin2011Extrapolation} and the references therein for eigenvalues of second order elliptic operators. For the biharmonic eigenvalue problem, the asymptotic expansions of eigenvalues by the Ciarlet-Raviart scheme and the nonconforming rotated $Q_1$ and the enriched rotated $Q_1$ elements on rectangular meshes were analyzed in  \cite{Chen2007Asymptotic} and \cite{Jia2010Approximation}, respectively.
For some nonconforming elements on triangular meshes, the lack of  this crucial superclose property  leads to a substantial difficulty of the asymptotic analysis. Until recently, \cite{hu2019asymptotic} proved the first optimal asymptotic result of two nonconforming elements for eigenvalues of the Laplacian operator. 

Asymptotically exact a posteriori error estimates is another efficient technique to improve the accuracy of eigenvalues. The key of such a posteriori error estimates is to express the error in terms of some computable high accuracy approximations. For eigenvalues of the Laplacian operator, \cite{hu2020asymptotically,naga2006enhancing} studied the asymptotically exact a posteriori error estimates for some conforming and nonconforming elements.

In this paper, we establish the first asymptotic analysis for eigenvalues by the Morley element and analyze the efficiency of the extrapolation algorithm. Inspired by \cite{hu2019asymptotic}, we overcome the difficulty caused by the lack of a crucial superclose property and get 
\begin{equation}\label{eq:intro} 
\lambda-\Mlam=\parallel (I - \PiHHJ)\nabla^2 u  \parallel_{0,\Om}^2 
+ 2I_1 + 2I_2 - 2I_3 + \cO(h^4|u|_{\frac{9}{2},\Om}^2),
\end{equation}   
where $(\Mlam, \Mu)$  is an eigenpair by the Morley element, the interpolation operator $\PiHHJ$ is defined in \eqref{def:fortin} and $I_1$, $I_2$, $I_3$ are defined in \eqref{crI}. To achieve an optimal result, we conduct a new technical analysis for each term on the right hand side of \eqref{eq:intro}. The analysis in this paper is quite different from the one in \cite{hu2019asymptotic}  because some natural orthogonal property is absent in this case.
We establish an explicit expression with a vanishing subdominant term for the interpolation in \cite{hu2012lower}. By use of this expression, we can cancel some suboptimal terms in $\parallel (I - \PiHHJ)\nabla^2 u  \parallel_{0,\Om}^2$ and get the desired optimal expansion.
By employing the commuting property and the equivalence with the HHJ element, we express $I_1$ in terms of  the second order accuracy term $\uHHJ$, instead of the first order accuracy term $\sHHJ$. In this way, we achieve an optimal estimate of $I_1$, where $\uHHJ$ and $\sHHJ$ are defined in \eqref{def:sigfeig}.
We express the consistency error term $I_2$ in terms of jumps along interior edges, which allows some cancellation and is the key to a desired optimal analysis.
For $I_3$, we establish an explicit expression of the interpolation error of the Morley element, and cancel the suboptimal terms between adjacent elements forming a parallelogram, which is crucial in getting the optimal result.

We also design an asymptotically exact a posteriori error estimate and the corresponding approximate eigenvalue by the Morley element. By a simple postprocessing technique, the accuracy of the approximate eigenvalue can be improved to $\mathcal{O}(h^3)$. Numerical results show that this postprocessing technique is effective on both uniform and adaptive meshes, and achieves better performance than the extrapolation method.

The remaining paper is organized as follows. Section 2 presents fourth order elliptic eigenvalue problems and some notations. Section 3 explores an optimal asymptotic expansion of approximate eigenvalues by the Morley element and analyzes the optimal convergence rate of eigenvalues by the extrapolation method. Section 4 proposes and analyzes an asymptotically exact a posterior error estimate of eigenvalues by the Morley element. Section 5 presents some numerical tests.

\section{Notations and Preliminaries}
\subsection{Notations}\label{sec:notation}
Given a nonnegative integer $k$ and a bounded domain $\Om\subset \mathbb{R}^2$ with boundary $\partial \Om$, let $W^{k,\infty}(\Om,\mathbb{R})$, $H^k(\Om,\mathbb{R})$, $\parallel \cdot \parallel_{k,\Om}$ and $|\cdot |_{k,\Om}$ denote the usual Sobolev spaces, norm, and semi-norm, respectively. The Sobolev spaces
$$
H_0^1(\Om,\mathbb{R}) = \{u\in H^1(\Om,\mathbb{R}): u|_{\partial \Om}=0\},\ H_0^2(\Om,\mathbb{R}) = \{u\in H^2(\Om,\mathbb{R}): u|_{\partial \Om}={\partial u\over \partial \bn}|_{\partial \Om}=0\}.
$$
Denote the standard $L^2(\Om,\mathbb{R})$ inner product and $L^2(K,\mathbb{R})$ inner product by $(\cdot, \cdot)$ and $(\cdot, \cdot)_{0,K}$, respectively.

Suppose that $\Om\subset \mathbb{R}^2$ is a convex polygonal domain and the partition $\cT_h$ of domain $\Om$ is assumed to be uniform in the sense that any two adjacent triangles form a parallelogram. Let $|K|$ denote the area of element $K$ and $|e|$ the length of edge $e$. Let $h_K$ denote the diameter of element $K\in \cT_h$ and $h=\max_{K\in\cT_h}h_K$. Denote the set of all interior edges and boundary edges of $\cT_h$ by $\cE_h^i$ and $\cE_h^b$, respectively, and $\cE_h=\cE_h^i\cup \cE_h^b$. 

Let element $K$ have vertices $\bold{p}_i=(p_{i1},p_{i2}),1\leq i\leq 3$ oriented counterclockwise, and corresponding barycentric coordinates $\{\psi_i\}_{i=1}^3$. Let $M_K=(M_1, M_2)$ denote  the centroid of the element, $\{e_i\}_{i=1}^3$ the edges of element $K$, $\{d_i\}_{i=1}^3$ the perpendicular heights,  $\{\theta_i\}_{i=1}^3$ the internal angles, $\{\bold{m}_i\}_{i=1}^3$ the  midpoints of edges
 $\{e_i\}_{i=1}^3$, and $\{\bold{n}_i\}_{i=1}^3$ the unit outward normal vectors, $\{\bold{t}_i\}_{i=1}^3$ the unit tangent vectors with counterclockwise orientation (see Fig.~\ref{fig:geometric}). 
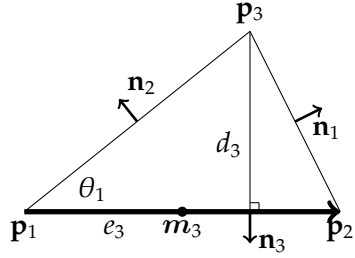
\begin{figure}[!ht]
\begin{center}
\begin{tikzpicture}[xscale=6,yscale=6]
\tikzstyle{every node}=[font=\Large,scale=0.9]
\draw[->][line width=2pt](0,0) -- (0.7,0);
\draw[-] (0,0) -- (0.5,0.4);
\draw[-] (0.7,0) -- (0.5,0.4);
\draw[-] (0.5,0) -- (0.5,0.4);
\draw[-] (0.5,0.02) -- (0.52,0.02);
\draw[-] (0.52,0.02) -- (0.52,0);
\draw[->][line width=1pt] (0.5,0) -- (0.5,-0.07);
\draw[->][line width=1pt] (0.6,0.2) -- (0.66,0.23);
\draw[->][line width=1pt] (0.25,0.2) -- (0.21,0.25);
\node[below] at (0,0) {$\bold{p}_1$};
\node[below] at (0.7,0) {$\bold{p}_2$};
\node[above] at (0.5,0.4) {$\bold{p}_3$};
\node[below] at (0.35,0) {$\bM_3$};
\draw [fill] (0.35,0) circle [radius=0.01];
\node[below] at (0.2,0) {$e_3$};
\node[right] at (0.5,-0.07) {$\bold{n}_3$};
\node[right] at (0.62,0.18) {$\bold{n}_1$};
\node at (0.26,0.28) {$\bold{n}_2$};
\node[left] at (0.5,0.15) {$d_3$};
\node[right] at (0.1,0.05) {$\theta_1$};
\end{tikzpicture}
\end{center}
\caption{Paramters associated with a triangle $K$.}
\label{fig:geometric}
\end{figure}
There hold the following relationships
$
d_i|e_i|=2|K|
$
and
\begin{equation}\label{nlambda}
\begin{split}
\nabla \psi_i=-\frac{\bold{n}_i}{d_i},\quad  
\sin \theta_i=\bn_{i-1}\cdot \bt_{i+1} = -\bn_{i+1}\cdot \bt_{i-1},
\\
\cos \theta_{i}
=-\boldsymbol{n}_{i-1} \cdot \boldsymbol{n}_{i+1}= {|e_{i-1}|^2 + |e_{i+1}|^2 - |e_{i}|^2\over 2|e_{i-1}||e_{i+1}|}
\end{split}
\end{equation}
among the quantities \cite{huang2008superconvergence}. 
For $K\subset\R^2,\ r\in \mathbb{Z}^+$, let $P_r(K, \R)$ be the space of all polynomials of degree not greater than $r$ on $K$. 
Denote the piecewise Hessian operator by $\nabla_h^2$. For any $1\leq i, j, k\leq 2$, denote the third order derivative $\frac{\partial^3 v}{\partial x_i\partial x_j\partial x_k}$ by $\partial_{ijk} v$. For any $\alpha=(\alpha_1, \alpha_2)$, denote
$$
|\alpha|=\alpha_1 + \alpha_2,\qquad \alpha !=\alpha_1!\alpha_2!,\qquad D^\alpha v=\partial_1^{\alpha_1}\partial_2^{\alpha_2}v.
$$
For ease of presentation,  the symbol $A\lesssim B$ will be used to denote that $A\leq CB$, where $C $ is a positive constant.
\subsection{Morley element for eigenvalue problems}\label{sec:model}

Consider the biharmonic eigenvalue problem for plate bending, which finds $\lambda$ and $\|u\|_{0, \Om}=1$ such that
\begin{equation}\label{model}
\Delta^2 u =\lambda u,\quad \text{ in }\Omega,
\end{equation}
with the clamped boundary condition
\begin{equation}\label{bc:cb}
u|_{\partial \Omega}={\partial u\over \partial \bn}\big|_{\partial \Om}=0,
\end{equation}
or the simply supported boundary condition
\begin{equation}\label{bc:ssc}
u|_{\partial \Omega}={\partial^2 u\over \partial \bn^2}\big|_{\partial \Om}=0.
\end{equation}
The weak formulation for \eqref{model} is to find $(\lambda, u)\in \mathbb{R}\times V$  such that  $\parallel u \parallel_{0,\Om}=1$ and
\be\label{variance}
a(u, v)=\lambda(u, v)\quad  \forall \ v\in V,
\ee
with 
$\displaystyle
a(w, v)=\int_{\Om} \nabla^2  w: \nabla^2 v \dx
$ and 
\begin{equation*}
V=\begin{cases}
H_0^2(\Om,\mathbb{R}) & \mbox{ for clamped boundary plates with }\eqref{bc:cb},
\\
H^2(\Om,\R)\cap H_0^1(\Om,\mathbb{R})& \mbox{ for simply supported plates with }\eqref{bc:ssc}.
\end{cases}
\end{equation*}
The bilinear form
$
a(\cdot, \cdot) 
$
is symmetric, bounded, and coercive, namely
$$
a(w,v)=a(v,w),\quad |a(w,v)|\lesssim \parallel w\parallel_{2,\Om}\parallel v\parallel_{2,\Om},\quad \parallel v\parallel_{2,\Om}^2\lesssim a(v,v),\quad \forall w, v\in V.
$$
The eigenvalue problem \eqref{variance} has a sequence of eigenvalues
$$0<\lambda^1\leq \lambda^2\leq \lambda^3\leq ...\nearrow +\infty,$$
and the corresponding eigenfunctions
$u^1, u^2, u^3,... ,$
with
$(u^i, u^j)=\delta_{ij}.
$

The nonconforming Morley element space  $\VM$ over $\cT_h$ is defined \cite{ming2006morley,morley1968triangular} by
\begin{equation*}
\begin{split}
\VM:=&\big \{v\in L^2(\Om,\R)\big|v|_K\in P_2(K, \R)\text{ for any }  K\in\cT_h, v\ \text{is continuous at each interior}\\
& \text{ vertex and vanishes on each boundary vertex},\ \int_e \big [{\partial v\over \partial n}\big ]\ds =0\text{ for any }  e\in \cE_h^i,\\
&\text{ and } \int_e {\partial v\over \partial n}\ds=0\text{ for any }  e\in \cE_h^b\big\}
\end{split}
\end{equation*} 
if the clamped boundary condition \eqref{bc:cb} is imposed, and 
\begin{equation*}
\begin{split}
\VM:=&\big \{v\in L^2(\Om,\R)\big|v|_K\in P_2(K, \R)\text{ for any }  K\in\cT_h, v\ \text{is continuous at each interior}\\
& \text{ vertex and vanishes on each boundary vertex},\ \int_e \big [{\partial v\over \partial n}\big ]\ds =0\text{ for any }  e\in \cE_h^i\big\}
\end{split}
\end{equation*}
if the simply supported boundary condition \eqref{bc:ssc} is imposed.
The corresponding canonical interpolation operator $\PiM: V\rightarrow \VM $ is defined by
\begin{equation}\label{ecrinterpolation}
\int_e{\partial \PiM v\over \partial n}\ds= \mathlarger{\int}_e {\partial v\over \partial n}\ds, \forall e\in\cE_h, \quad  \PiM v(\bp)= v(\bp)\quad\text{ for any vertex }  \bp.
\end{equation}
The corresponding finite element approximation of \eqref{variance} is to find $(\Mlam, \Mu)\in \mathbb{R}\times \VM$  such that $\parallel \Mu\parallel_{0,\Om}=1$ and
\be\label{discrete}
a_h(\Mu,v_h)=\Mlam(\Mu, v_h)\quad \forall \ v_h\in \VM,
\ee
with the discrete bilinear form
$
\displaystyle a_h(w_h,v_h):=\sum_{K\in\cT_h}\int_K \nabla_h^2 w_h: \nabla_h^2 v_h\dx. 
$
Denote the approximate solution of \eqref{discrete} by $(\Mlam^i,\Mu^i)$ with 
$
0<\Mlam^1\le \Mlam^2\le \cdots \nearrow  \Mlam^{N_V},
$
where $N_V={\rm dim }\VM$ and $(\Mu^i, \Mu^j)=\delta_{ij}$, $1\le i,\ j\le N_V$.

Let $\lambda$ be an eigenvalue of Problem \eqref{variance} with multiplicity $q$ and 
$$
M(\lambda)=\{w\in  V: w \mbox{ is an eigenvector of Problem \eqref{variance} corresponding to}\ \lambda\}.
$$ 
Without loss of generality, assume the index of $\lambda$ are $k_0+1, \cdots, k_0+q$, that is,
$
\lambda^{k_0}<\lambda=\lambda^{k_0+1}=\cdots = \lambda^{k_0+q}<\lambda^{k_0+q+1}.
$
Denote
\begin{align*}
M_h(\lambda) &= {\rm span}\{\Mu^{k_0+1},\ \Mu^{k_0+2}, \cdots,\ \Mu^{k_0+q}\}\subset \VM.
\end{align*}
Suppose that $(\Mlam, \Mu)$ is the $i$-th eigenpair of Problem \eqref{discrete} by the Morley element, the theory of nonconforming eigenvalue approximations, see for instance, 
\cite{babuvska1987estimates,babuvska1989finite,boffi2017posteriori,Hu2014Lower,rannacher1979nonconforming} 
and the references therein, indicates that  there exists $u\in M(\lambda)$ with $\lambda= \lambda^i$ such that
\begin{equation}\label{M:est}
\begin{split}
h|\lambda-\Mlam|+ h^{j}|u- \PiM u|_{j, h}+ h\parallel u- \Mu\parallel_{0,\Om} 
+ h^j\parallel \nabla_h^j (u-\Mu)\parallel_{0,\Om}&\lesssim h^{3}\parallel u\parallel_{3,\Om}
\end{split}
\end{equation} 
provided that the domain is convex and $M(\lambda)\subset  V\cap H^3(\Om)$, where $j=1$, $2$. Whenever there is no ambiguity,  $(\lambda, u)$ defined this way is the called the corresponding eigenpair to $(\Mlam, \Mu)$ of Problem \eqref{discrete} if the estimate  \eqref{M:est} holds. 

For the Morley element, there holds the following commuting property \cite{Crouzeix1973Conforming,Hu2014Lower}
\begin{equation}\label{commuting} 
\int_K \nabla^2(w - \PiM w): \nabla^2  v_h \dx =0\quad \forall\ w\in V, v_h\in \VM.
\end{equation} 
This, together with the technique in \cite{hu2020asymptotically,hu2019asymptotic}, guarantees the following expansion
\begin{equation}\label{commutId}
\lambda-\Mlam =\|\nabla_h^2 (u - \Mu)\|_{0, \Om}^2-2\lambda (u-\PiM u, u ) + \cO(h^4|u|_{3, \Om}^2).
\end{equation}
The asymptotic expansion in this paper is based on this crucial identity \eqref{commutId}. 

\subsection{Hellan--Herrmann--Johnson element for source problems}
For any source term $f$, the plate bending problem seeks $u^f\in V$ such that
\be\label{mixvariance}
\Delta^2 u^f=f
\ee
with boundary condition \eqref{bc:cb} or \eqref{bc:ssc}. Define 
\begin{equation*} 
D=\{v\in H^1_0(\Om, \R): v|_K\in H^2(K, \R)\}
\end{equation*}
and the space for  an auxiliary variable $\sigma^f :=\nabla^2 u^f$  by 
\begin{equation*} 
S=\{\tau\in L^2(\Om, \S): \tau|_K\in H^1(K, \S), \text{ and } M_{\bn\bn}(\tau) \text{ is continuous across interior edges}\}
\end{equation*}
with $\S:=\text{symmetric }\Rmath^{2\times2}$ if the clamped boundary condition \eqref{bc:cb} is imposed and 
\begin{equation*} 
\begin{split}
S=\{\tau\in L^2(\Om, \S): &\ \tau|_K\in H^1(K, \S), \text{ and } M_{\bn\bn}(\tau) \text{ is continuous across interior edges,}
\\
&\ M_{\bn\bn}(\tau)=0  \text{ on boundary edges}\}
\end{split}
\end{equation*}
if the simply supported boundary condition \eqref{bc:ssc} is imposed.

Given $K\in\Tma$ and $\tau\in H^1(K,\S)$, let 
$$
M_{\bn\bn}(\tau)=\bn^T\tau \bn,\quad M_{\bn\bt}(\tau)=\bt^T\tau \bn
$$
with the unit outnormal  $\bn$ and unit tangential direction $\bt$ with counterclockwise orientation of $\partial K$. 
Since $\sigma^f\bn : \nabla v=M_{\bn\bn}(\sigma^f)\frac{\partial v}{\partial \bn} + M_{\bn\bt}(\sigma^f)\frac{\partial v}{\partial \bt}$, the integration by parts gives
\begin{equation}
(f, v)= (\sigma^f, \nabla^2 v) + \sum_{K\in\cT_h}\int_{\partial K} (\nabla\cdot \sigma^f)\cdot \bn v - M_{\bn\bn}(\sigma^f){\partial v\over \partial \bn} - M_{\bn\bt}(\sigma^f)\frac{\partial v}{\partial \bt} \ds.
\end{equation}
Note that $v\in D$ is continuous on interior edges and zero on the boundary $\partial \Omega$, so is the tangential derivative of $v$. Thus,
\begin{equation}
(f, v)= (\sigma^f, \nabla^2 v) -  \sum_{K\in\cT_h}\int_{\partial K} M_{\bn\bn}(\sigma^f){\partial v\over \partial \bn} \ds.
\end{equation}
The mixed formulation of source problem \eqref{mixvariance}, which was analyzed in  \cite{johnson1973convergence}, seeks $(\sigma^f, u^f)\in S\times D$ such that
\begin{equation}\label{mixmodel}
\begin{aligned}
(\sigma^f, \tau) + \sum_{K\in\cT_h} -(\tau, \nabla^2 u^f)_{0, K} + \int_{\partial K} M_{\bn\bn}(\tau){\partial u^f\over \partial \bn} \ds&=0,\ &\forall \tau \in S,
\\
\sum_{K\in\cT_h} -(\sigma^f, \nabla^2 v)_{0, K} + \int_{\partial K} M_{\bn\bn}(\sigma^f){\partial v\over \partial \bn} \ds&=-
(f,v), \ &\forall v\in D.
\end{aligned}
\end{equation} 
Define the discrete spaces \cite{arnold2020hellan,chen2018multigrid} 
$$
\UHHJ:=\big \{v\in D: v|_K\in P_1(K, \R)\text{ for any }K\in \cT_h\big \},
$$ 
$$
\SHHJ:= 
\big \{\tau\in S: \tau|_K\in P_0(K, \S)\text{ for any }K\in \cT_h\big \}.
$$
The first order HHJ element \cite{chen2018multigrid,johnson1973convergence} of Problem \eqref{mixmodel} seeks $(\sHHJf, \uHHJf)\in \SHHJ\times \UHHJ$ such that  
\begin{equation}\label{eq:HHJ}
\begin{aligned}
(\sHHJf, \tau_h) + \sum_{e\in\cE_h}  \int_{e} M_{\bn\bn}(\tau_h)[{\partial \uHHJf\over \partial \bn}] \ds&=0,\ &\forall \tau_h \in \SHHJ,
\\
\sum_{e\in\cE_h}  \int_{e} M_{\bn\bn}(\sHHJf)[{\partial v_h\over \partial \bn}] \ds&=-(f,v_h), \ &\forall v_h\in \UHHJ.
\end{aligned}
\end{equation}
It follows from the theory of mixed finite element methods \cite{comodi1989hellan} that
\be\label{RT:est}
\parallel u^f- \uHHJf\parallel_{0,\Om} + h\parallel \sigma^f -\sHHJf\parallel_{0,\Om}+h |u^f- \uHHJf |_{1,\Om}\lesssim h^{2}\parallel u^f\parallel_{3,\Om},
\ee
provided that $u^f\in V\cap H^{3}(\Om,\mathbb{R})$.  In this paper, we consider two different source terms $f=\lambda u$ and $f=\Mlam \Mu$. Let $(\sHHJc, \uHHJc)$ and $ (\sHHJ, \uHHJ)\in \SHHJ\times \UHHJ$ be the solutions of Problem \eqref{eq:HHJ} with source terms $f=\lambda u$ and $f=\Mlam \Mu$, respectively. 
Then,
\begin{equation}\label{def:sigfeig}
\sHHJ= \sigma_{\rm HHJ}^{\Mlam\Mu},\quad \uHHJ= u_{\rm HHJ}^{\Mlam\Mu}.
\end{equation}
In the rest of this paper, denote $\sigma=\nabla^2 u$ where $u$ is the eigenfunction of Problem \eqref{model}.

Define the interpolation operator $\PiHHJ: S\rightarrow \SHHJ$ by
\begin{equation}\label{def:fortin}
\int_e M_{\bn\bn}(\PiHHJ \tau)\ds=\int_e M_{\bn\bn}(\tau)\ds\text{\quad for any }e\in \cE_h, \bold{\tau}\in S.
\end{equation}
There exists the following identity in \cite{Hu2015The} that
\begin{equation}\label{eq:HHJid}
(\sHHJc - \PiHHJ \sigma, \sHHJc - \sigma)=0.
\end{equation}

For the plate bending problem with the clamped boundary condition \eqref{bc:cb}, it was analyzed in \cite{hu2021optimal} that the HHJ element admits an important superconvergence property on uniform triangulations as presented below. For Problem \eqref{model} with the simply supported condition \eqref{bc:ssc}, $M_{\bn\bn}(\sHHJc - \PiHHJ \sigma)=0$ for all edges $e\in \cE_h^b$. A simple extension of the analysis in \cite{Hu2016Superconvergence} proves the  superconvergence property of the HHJ element when \eqref{bc:ssc} is imposed.
\begin{lemma}\label{Lm:super}
Suppose that $(\sHHJc,\uHHJc)$ is the solution of Problem \eqref{eq:HHJ} with boundary condition \eqref{bc:cb} or \eqref{bc:ssc} on a uniform triangulation, $f=\lambda u$ and $u\in V\cap H^{r}(\Om,\mathbb{R})$. It holds that
\begin{equation}\label{eq:super}
\parallel \sHHJc-\PiHHJ\sigma\parallel_{0,\Om}\lesssim h^2 \big (| u|_{r,\Om}+\kappa |\ln h|^{1/2}|u|_{3,\infty,\Om}\big ).
\end{equation}
where $\kappa$ is defined in \eqref{kappadef},  $r=\frac{9}{2}$ if boundary condition \eqref{bc:cb} is imposed, and $r=\frac{7}{2}$ if \eqref{bc:ssc} is imposed.
\end{lemma}

\subsection{Equivalence between the HHJ element and the Morley element}
Let $\Pi_{\rm D} v$ be the linear interpolation of any function $v\in V+ \VM$. 
Given a function $f\in L^2(\Omega, \R)$. Let $u_{\rm M}^{f}\in \VM$ be the Morley solution of the source problem
\begin{equation}\label{eq:Mc}
a_M(u_{\rm M}^f, v_h)=(f,v_h),\quad \forall v_h\in \VM,
\end{equation} 
and $\widetilde{u}_{\rm M}^f\in \VM$ be the modified Morley solution of the source problem
\begin{equation}\label{eq:Mmc}
a_M(\widetilde{u}_{\rm M}^f, v_h)=(f,\Pi_{\rm D}v_h),\quad \forall v_h\in \VM.
\end{equation} 
As analyzed in \cite{arnold1985mixed}, the HHJ solution of source problem \eqref{eq:HHJ} and the modified Morley solution of source problem \eqref{eq:Mmc} are equivalent in the following sense
\begin{equation}
\sHHJf = \nabla_h^2\widetilde{u}_{\rm M}^f,\quad \uHHJf=\Pi_{\rm D}\widetilde{u}_{\rm M}^f.
\end{equation}
This equivalence leads to the following lemma.
\begin{lemma}\label{lm:equiv}
Let $(\sHHJc, \uHHJc)$, $(\sHHJ, \uHHJ)$ be the solutions of Problem \eqref{eq:HHJ} with $f=\lambda u$ and $f=\Mlam \Mu$, respectively,    $\Mumc$ and $\Mum$ be the solutions of Problem \eqref{eq:Mmc} with $f=\lambda u$ and $f=\Mlam \Mu$, respectively. It holds that
\begin{equation}\label{relation}
\sHHJc=\nabla^2_{h}\Mumc, \quad\ \uHHJc=\Pi_{\rm D}\Mumc,\qquad 
\sHHJ=\nabla^2_{h}\Mum, \quad\ \uHHJ=\Pi_{\rm D}\Mum.
\end{equation}  
\end{lemma}


\section{Optimal Analysis of Extrapolation Algorithm for the Morley element}\label{sec:CR}

In this section, we consider the extrapolation algorithm on the eigenvalues by the Morley element. 
An  asymptotic expansion of eigenvalues is established, which gives an optimal theoretical analysis of the extrapolation algorithm on the eigenvalue problem.

Given the approximate eigenvalues  $\Mlam^h$ and $\Mlam^{2h}$ on  $\cT_h$ and   $\cT_{2h}$, respectively. The extrapolation algorithm computes a new approximate eigenvalue by 
\begin{equation}\label{exmethod}
\lambda_{\rm M}^{\rm EXP}={2^{\alpha} \Mlam^{h}-\Mlam^{2h}\over 2^{\alpha}-1},
\end{equation}
where the convergence rate $\alpha=2$ if the eigenfunction is smooth enough.
Suppose that there exists such an asymptotic expansion of eigenvalues
\begin{equation}\label{exid}
\lambda-\Mlam^{h}=Ch^\alpha + \cO(h^{\beta}) \text{ with }\beta>\alpha,
\end{equation}
where $C$ is independent on the mesh size $h$. It is easy to verify that the extrapolation algorithm \eqref{exmethod} improves the accuracy of eigenvalues to $\mathcal{O}(h^\beta)$ if  \eqref{exid} holds.

In the rest of this section, we establish
 an expansion in the form of \eqref{exid}  with an optimal rate $\beta=4$ and $C$ is expressed explicitly by the function $u$.

\subsection{Error expansions for eigenvalues}
The classic asymptotic analysis does not work for the Morley element because of the lack of a crucial superclose property. Inspired by \cite{hu2019asymptotic}, we use the equivalence between the mixed HHJ element and the modified Morley element in Lemma \ref{lm:equiv} and the superconvergence property of the mixed HHJ element in Lemma \ref{Lm:super} to establish an asymptotic expansion of eigenvalues by the Morley element.  



To begin with, we list the discrete problems and  the corresponding solutions considered in this paper:
\begin{enumerate}
\item[(P1)] eigenvalue problem \eqref{discrete} by the Morley element: $\Mu$;

\item[(P2)] source problem \eqref{eq:Mc}  by the Morley element: $\Muc$;

\item[(P3)] source problem \eqref{eq:Mmc} by the modified Morley element: $\Mumc$ and $\Mum$;

\item[(P4)] source problem \eqref{eq:HHJ} by the HHJ element: $(\sHHJc, \uHHJc)$, $ (\sHHJ, \uHHJ)$.
\end{enumerate}
The Morley solution $\Mu$ in Problem (P1) is also a solution of Problem (P2) with $f=\Mlam\Mu$. The following Lemma \ref{supereig1} analyzes some superclose property of the Morley element and the mixed HHJ element, including the relation between Problems (P1) and (P2), and that between Problems (P2) and (P3).
Lemma \ref{lm:equiv} presents the equivalence between the solutions of Problems (P3) and (P4). Thus, the superconvergence in Lemma \ref{Lm:super} of the HHJ element for Problem (P4) can be used to analyze the expansion of eigenvalues of Problem (P1).

\begin{lemma}\label{supereig1}
Suppose that $(\lambda, u)$ is an eigenpair  of Problem \eqref{variance}, $(\sHHJc, \uHHJc)$ and $(\sHHJ, \uHHJ)$ are the solutions of Problems \eqref{eq:HHJ} with $f=\lambda u$ and $f=\Mlam \Mu$, respectively.  It holds that
\begin{equation}\label{eq:Msuper} 
|\Muc-\Mumc|_{2,h} + | \Mu-\Mum|_{2,h} + |\Muc-\Mu|_{2,h} + \parallel \sHHJc - \sHHJ \parallel_{0,\Om}\lesssim h^2\|u\|_{3,\Om}, 
\end{equation}
provided that $u\in V\cap H^3(\Om,\R)$.
\end{lemma}
\begin{proof}
To bound $\|\nabla_{h}^2(\Muc-\Mumc)\|_{0,\Omega}$, let $v_h=\Muc-\Mumc$ in \eqref{eq:Mc} and \eqref{eq:Mmc}. It holds that
$$
\|\nabla_{h}^2(\Muc-\Mumc)\|_{0,\Omega}^2 = (\lambda u, (I-\Pi_{\rm D})(\Muc-\Mumc))
\lesssim \lambda h^2 \|\nabla_{h}^2(\Muc-\Mumc)\|_{0,\Omega},
$$
which implies that 
$
\|\nabla_{h}^2(\Muc-\Mumc)\|_{0,\Omega}\lesssim h^2,
$
and completes the estimate of the first term on the left-hand side of \eqref{eq:Msuper}. A similar analysis leads to the following estimate of the second term, that is
$
\|\nabla_{h}^2(\Mu-\Mum)\|_{0,\Omega} \lesssim h^2.
$
A similar analysis to the one for Lemma  3.1 in \cite{hu2019asymptotic} gives
$
\|\nabla_{h}^2(\Muc-\Mu)\|_{0,\Omega}\lesssim h^2\|u\|_{3,\Om}
$
for the third term.
Consider the last term on the left-hand side of \eqref{eq:Msuper}. By the equivalence \eqref{relation} between the modified Morley element and the HHJ element,
$$
\sHHJc - \sHHJ = \nabla_h^2 (\Mumc - \Mum) = \nabla_h^2 (\Mumc -  \Muc + \Muc - \Mu + \Mu -\Mum).
$$
It follows that
$
\parallel \sHHJc - \sHHJ \parallel_{0,\Om}\lesssim h^2\|u\|_{3,\Om},
$
which completes the proof.
\end{proof}

For simplicity of presentation, we introduce the following notations
\begin{equation}\label{crI}
\begin{array}{llll}
I_1&=(\sigma - \sHHJc, \sHHJc - \sHHJ) ,&I_2&=(\sigma - \sHHJ, \nabla_h^2 (\Mum-  \Mu)), \\
 I_3&=\lambda (u-\PiM u, u ).&&
\end{array}
\end{equation}
The asymptotic  expansion of eigenvalues by the Morley element is
 based on the decomposition in the following theorem.
\begin{Th}\label{Th:extraCR}
Suppose that $(\lambda , u )$ is the solution of Problem \eqref{variance} with $u \in V\cap H^{\frac{9}{2}}(\Om,\mathbb{R})$, and $(\Mlam, \Mu)$ is the corresponding discrete eigenpair  of Problem \eqref{discrete} by the Morley element.   If the triangulation is uniform, it holds that
\begin{equation}\label{eq:Mdeco} 
\lambda-\Mlam=\parallel (I - \PiHHJ)\nabla^2 u  \parallel_{0,\Om}^2 
+ 2I_1 + 2I_2 - 2I_3 + \cO(h^4|\ln h||u|_{\frac{9}{2},\Om}^2),
\end{equation}    
where $I_1$, $I_2$ and $I_3$ are defined in \eqref{crI}.
\end{Th}
\begin{proof}
Recall the expansion \eqref{commutId} of eigenvalues by the Morley element  
\begin{equation}\label{eq:deco0} 
\lambda-\Mlam =\|\nabla_h^2 (u - \Mu)\|_{0, \Om}^2-2\lambda (u-\PiM u, u ) + \cO(h^4|u|_{3, \Om}^2). 
\end{equation} 
Thanks to the equivalence $\sHHJ= \nabla_h^2\Mum$ in \eqref{relation},
\ben
\nabla^2 u-\nabla_h^2 \Mu= (\nabla^2 u -\PiHHJ\nabla^2 u) + (\PiHHJ\nabla^2 u - \sHHJc) + (\sHHJc - \sHHJ) +  \nabla_h^2 (\Mum-  \Mu),
\een
with the HHJ solutions $\sHHJc$ and $\sHHJ$ of the source problem \eqref{eq:HHJ} with $f=\lambda u$ and $f=\Mlam \Mu$, respectively,  and the Morley solutions $\Mu$ and $\Mum$ of the eigenvalue problem \eqref{discrete} and the modified problem \eqref{eq:Mmc}, respectively. It holds that
\begin{equation} \label{Mtotal}
\begin{split}
\lambda-\Mlam=&\parallel (I - \PiHHJ)\sigma  \parallel_{0,\Om}^2 
+ \parallel \PiHHJ\sigma - \sHHJc \parallel_{0,\Om}^2
+ \parallel \sHHJc - \sHHJ \parallel_{0,\Om}^2 
\\
&+\parallel  \nabla_h^2 (\Mum-  \Mu)\parallel_{0,\Om}^2 
+2(\sigma -\PiHHJ\sigma, \PiHHJ\sigma - \sHHJc)
\\
&
+2(\sigma -\PiHHJ\sigma, \sHHJc - \sHHJ)
+ 2(\sigma -\PiHHJ\sigma, \nabla_h^2 (\Mum-  \Mu))
\\
&
+ 2(\PiHHJ\sigma - \sHHJc, \sHHJc - \sHHJ)
+ 2(\PiHHJ\sigma - \sHHJc, \nabla_h^2 (\Mum-  \Mu))
\\
& 
+ 2(\sHHJc - \sHHJ, \nabla_h^2 (\Mum-  \Mu))
-2\lambda (u-\PiM u, u )
+\cO(h^4|u|_{3,\Om}^2).
\end{split}
\end{equation} 
Recall the superconvergence \eqref{eq:super} of the HHJ element in Lemma \ref{Lm:super} and the superclose property \eqref{eq:Msuper} of both the Morley element and the HHJ element. Then, 
\begin{equation}\label{eq:term0}
\begin{aligned}
\parallel \PiHHJ\sigma - \sHHJc\parallel_{0,\Om}^2 
+ \parallel \sHHJc - \sHHJ \parallel_{0,\Om}^2
 + \|\nabla_{h}^2(\Mu-\widetilde{u}_{\rm M})\|_{0,\Omega}^2
&
\\
+ 2(\PiHHJ\sigma - \sHHJc, \sHHJc - \sHHJ)
+ 2(\PiHHJ\sigma - \sHHJc, \nabla_h^2 (\Mum-  \Mu))&
\\
+ 2(\sHHJc - \sHHJ, \nabla_h^2 (\Mum-  \Mu))
&\lesssim h^4 |\ln h|\| u\|_{\frac{9}{2},\Om}^2.
\end{aligned}
\end{equation} 
A combination of  \eqref{eq:HHJid} and the superconvergence property \eqref{eq:super}  of the HHJ element yields
\begin{equation}\label{eq:term1}
\big |(\sigma -\PiHHJ\sigma, \PiHHJ\sigma - \sHHJc)\big | =\big | (\sHHJc -\PiHHJ\sigma, \PiHHJ\sigma - \sHHJc)\big |\lesssim h^4|\ln h| | u|_{\frac{9}{2},\Om}^2.
\end{equation} 
A substitution of \eqref{eq:term0} and \eqref{eq:term1} into \eqref{Mtotal} leads to 
\begin{equation*} 
\lambda-\Mlam=\parallel (I - \PiHHJ)\nabla^2 u  \parallel_{0,\Om}^2 
+ 2I_1 + 2I_2 - 2I_3 + \cO(h^4|\ln h||u|_{\frac{9}{2},\Om}^2),
\end{equation*}    
which completes the proof.
\end{proof}

In the rest of this section, we will conduct an asymptotic analysis of each term on the right-hand
 side of the expansion \eqref{eq:Mdeco} in the above theorem.

\subsection{Asymptotic expansion of $\parallel (I-\PiHHJ)\nabla^2 u \parallel_{0,\Om}^2$}\label{sec:rt2}

Define the three basis functions for the HHJ element
\begin{align*}
\phiH^i(\bx)=-{1\over 2\sin\theta_{i-1}\sin\theta_{i+1}}\big (\bold{t}_{i-1}\bold{t}_{i+1}^T + \bold{t}_{i+1}\bold{t}_{i-1}^T \big ) \in P_0(K, \S),\qquad  1\le i\le 3.
\end{align*}
By \eqref{nlambda}, it is easy to verify that
\be\label{eq:HHJbase} 
{1\over |e_j|}\int_{e_j}\bn_j^T\phiH^i\bn_j \ds= \delta_{ij}.
\ee 
Define  four short-hand notations for the HHJ element
\begin{equation}\label{HHJphi}
\begin{split}
\phi_1(\bx)&={1\over 6|K|^{1/2}}(x_1-M_1)^3,\quad \phi_2(\bx)={1\over 2|K|^{1/2}}(x_1-M_1)^2(x_2-M_2)
\\
\phi_3(\bx)&={1\over 2|K|^{1/2}}(x_1-M_1)(x_2-M_2)^2,\quad \phi_4(\bx)={1\over 6|K|^{1/2}}(x_2-M_2)^3.
\end{split}
\end{equation}
Note that $\{\phi_i\}_{i=1}^4$ are linear independent and 
\begin{equation} \label{p3deco}
P_3(K, \R)=P_2(K, \R)\cup {\rm span}\{\phi_i: 1\le i\le 4\},
\end{equation}
\begin{equation} \label{phidelta}
\|\partial_{111}\phi_i\|_{0, K}=\delta_{1i},\ \|\partial_{112}\phi_i\|_{0, K}=\delta_{2i},\ \|\partial_{122}\phi_i\|_{0, K}=\delta_{3i}, \ \|\partial_{222}\phi_i\|_{0, K}=\delta_{4i}. 
\end{equation}
Define
\begin{equation}\label{RTcdef}
\HHJcro^{ij} = \frac{1}{|K|}\int_K \big((I-\PiHHJ)\nabla^2\phi_i\big )^T(I-\PiHHJ)\nabla^2\phi_j\dx,\quad 1\le i, j\le 4.
\end{equation}
\begin{lemma}\label{lm:gammaconstant}
Constants $\HHJcro^{ij} $ in \eqref{RTcdef} are the same on different elements of a uniform triangulation and  independent of the mesh size $h$.
\end{lemma} 
\begin{proof}
By the definition of $\PiHHJ$ and \eqref{eq:HHJbase},
\begin{equation}\label{HHJbasis}
\PiHHJ \nabla^2\phi_i = \sum_{j=1}^3 a_{\rm HHJ}^{ij}\phiH^j\quad \mbox{ with }\quad a_{\rm HHJ}^{ij} = {1\over |e_j|}\int_{e_j} \bn_j^T\nabla^2\phi_i \bn_j \ds.
\end{equation}
Since $\nabla^2 \phi_1\in P_1(K, \S)$,
$$
a_{\rm HHJ}^{1j} = \bn_j^T\nabla^2\phi_1(\bM_j) \bn_j 
= {1\over |K|^{1/2}}\bn_j^T\begin{pmatrix}
\bM_{j1} - M_1&0\\
0&0
\end{pmatrix} \bn_j 
$$ 
is constant independent of the mesh size $h$, where $\bM_j=(\bM_{j1}, \bM_{j2})$ is the midpoint of edge $e_j$. Similarly, for any $1\le i\le 4$ and $1\le j\le 3$, constant $a_{\rm HHJ}^{ij}$ is independent on the mesh size $h$. It follows from \eqref{RTcdef} and \eqref{HHJbasis} that
\begin{equation}\label{eq:gammadeco}
\begin{split}
\HHJcro^{ij} =& {1\over |K|} (\nabla^2\phi_i, \nabla^2\phi_j)_{0, K} - {1\over |K|} \sum_{k=1}^3 (\phiH^{k}, a_{\rm HHJ}^{ik}\nabla^2\phi_j + a_{\rm HHJ}^{jk} \nabla^2\phi_i)_{0,K} 
\\
& + {1\over |K|} \sum_{k, l=1}^3 a_{\rm HHJ}^{ik}a_{\rm HHJ}^{jl}(\phiH^{k}, \phiH^{l})_{0,K}.
\end{split}
\end{equation}
By the definition of $\phi_i\in P_3(K,\R)$ and $\phiH^i\in P_0(K,\R)$, each entry of $\nabla^2 \phi_i$  is a linear combination of $x_1-M_1$ and $x_2-M_2$. Thus,
$
(\phiH^{k}, \nabla^2\phi_i)_{0,K}=0
$
for any $1\le i\le 4$, $1\le k\le 3$.
For any $1\le i, j\le 4$, both
$
{1\over |K|}(\nabla^2\phi_i, \nabla^2\phi_j)_{0, K}$ and $ {1\over |K|}(\phiH^{i}, \phiH^{j})_{0,K}
$
on uniform triangulations
are constant independent of $h$. It follows \eqref{eq:gammadeco}   that  $\HHJcro^{ij}$ in \eqref{RTcdef} are the same on different elements and independent of  $h$.
\end{proof}

For any region $G$, define
\begin{equation} \label{def:F}
\begin{split}
F(u, G) & = \HHJcro^{11}\|\partial_{111}u\|_{0, G}^2 + \HHJcro^{22}\|\partial_{112}u\|_{0, G}^2 + \HHJcro^{33}\|\partial_{122}u\|_{0, G}^2 + \HHJcro^{44}\|\partial_{222}u\|_{0, G}^2
\\
 + &2\HHJcro^{12}\int_G \partial_{111}u\partial_{112}u\dx +2 \HHJcro^{13} \int_G \partial_{111}u\partial_{122}u\dx+ 2\HHJcro^{14} \int_G \partial_{111}u\partial_{222}u\dx
\\
 + & 2\HHJcro^{23}\int_G \partial_{112}u\partial_{122}u\dx +2 \HHJcro^{24}\int_G \partial_{112}u\partial_{222}u\dx +2\HHJcro^{34}\int_G \partial_{122}u\partial_{222}u\dx.
\end{split}
\end{equation}

The following lemma presents the Taylor expansion of the interpolation error  of the HHJ element.
\begin{lemma}\label{RT2h40}
For any $w\in P_3(K, \R)$,
\begin{equation}\label{identity:RT2} 
\parallel (I-\PiHHJ)\nabla^2 w\parallel_{0,K}^2=F(w, K) |K|.
\end{equation} 
\end{lemma}
\begin{proof}
For any $w\in P_3(K, \R)$, it follows from \eqref{p3deco} and \eqref{phidelta} that there exist $p_2\in P_2(K, \R)$ and constants $\{a_i\}_{i=1}^4$ such that 
\begin{equation}\label{expan:rt0}
w = \sum_{i=1}^4 a_i\phi_i + p_2,\quad  (I-\PiHHJ)\nabla^2 w=\sum_{i=1}^4a_i(I-\PiHHJ)\nabla^2\phi_i,
\end{equation} 
where 
\begin{equation}\label{expan:coeff}
a_1=|K|^{1/2}\partial_{111} w , \ a_2=|K|^{1/2}\partial_{112} w, \ a_3=|K|^{1/2}\partial_{122} w, \ a_4=|K|^{1/2}\partial_{222} w. 
\end{equation} 
A substitution of \eqref{RTcdef} to \eqref{expan:rt0} leads to 
\begin{equation}\label{identity:RT20} 
\parallel (I-\PiHHJ)\nabla^2 w\parallel_{0,K}^2= \sum_{i,j=1}^4a_ia_j\HHJcro^{ij} |K|=F(w, K) |K|, 
\end{equation}
which completes the proof.
\end{proof}

Lemma \ref{RT2h40} indicates that $F(u,\Omega)|K|$ is an expansion of $\| (I-\PiHHJ)\nabla^2 u\|_{0, \Om}^2$ with accuracy $\mathcal{O}(h^3)$. Next we improve this estimate  to an optimal rate $\mathcal{O}(h^4)$.
Define the interpolation $\Pik_K v \in P_l(K, \R)$ in \cite{hu2012lower} for each positive integer $l$ by
\begin{equation}\label{def:pik}
\int_K D^\alpha \Pik_K v\dx = \int_K D^\alpha v\dx\quad \mbox{ with }\quad |\alpha|\le l,
\end{equation}
Let $\Pi_h^lv|_K = \Pik_Kv$. There exists the following  error estimate of the interpolation error 
\begin{equation}\label{interr}
|(I-\Pik_K) v|_{m, K}\lesssim h^{l -m+1} |v|_{l + 1, K}, \quad \forall\ 0 \leq m\leq l+1.
\end{equation}
Note that
\begin{equation}\label{orthexpan}
\begin{aligned}
\|(I-\PiHHJ)\nabla^2 u\|_0^2
&=\|(I-\PiHHJ)\nabla_h^2\Pi^3_hu\|_0^2+\|(I-\PiHHJ)\nabla_{h}^2(I-\Pi^3_h)u\|^2_0 \\
&
+2((I-\PiHHJ)\nabla_{h}^2(I-\Pi^3_h)u,(I-\PiHHJ)\nabla_{h}^2\Pi_h^3u),
\end{aligned}
\end{equation}
where the second term on the right-hand
 side is a higher order term.
The key to analyze $\|(I-\PiHHJ)\nabla u\|_0^2$ is to prove a nearly orthogonal property of  
$((I-\PiHHJ)\nabla_{h}^2(I-\Pi^3_h)u,(I-\PiHHJ)\nabla_{h}^2\Pi_h^3u)$.
To this end, define a set of polynomials
\begin{equation}\label{eq:phialpha}
\begin{split}
&\phi_{(0,0)}=1,\quad \phi_{(1,0)}=x_1-M_1, \quad \phi_{(0,1)}=x_2 - M_2,
\\
&\phi_\alpha={1\over \alpha !}(x - M_K)^\alpha - \sum_{|\beta|\le |\alpha|-2} C_\alpha^\beta\phi_\beta, \quad \mbox{for}\quad  |\alpha|\ge 2.
\end{split}
\end{equation}
with constant 
\be\label{eq:Cdef}
C_{\alpha}^\beta={1\over \alpha !|K|}\int_K D^\beta (x- M_K)^\alpha\dx.
\ee
For any $|\alpha|=k$,  $\phi_\alpha\in P_{k}(K, \R)$  and it is the first term ${1\over \alpha !}(x - M_K)^\alpha$ that determines the $k$-th derivatives of $\phi_\alpha$.

The explicit expression of the interpolation $\Pik_K u$ in the following lemma admits an important property that all $\phi_\beta$ with $|\beta|=|\alpha|-1$ vanish
 in $\phi_\alpha$.
\begin{lemma}
For any nonnegative integer $l$ and $u\in H^l(K, \R)$, 
\begin{equation}\label{eq:pil}
\Pik_K u=\sum_{|\alpha|\le l} a_K^\alpha \phi_\alpha,\quad \mbox{ with }\quad
a_K^\alpha={1\over |K|}\int_K  D^\alpha u\dx.
\end{equation}  
Moreover,
\begin{equation}\label{eq:C0}
(I- \Pik_K)\Pi_K^{l+1} u=\sum_{|\alpha|= l+1} a_K^\alpha \phi_\alpha.
\end{equation} 
\end{lemma}
\begin{proof} 
First we prove that basis functions $\phi_\alpha$ in \eqref{eq:phialpha} satisfy that
\be\label{eq:huinterpro}
{1\over |K|}\int_K D^\gamma \phi_\alpha \dx = \delta_{\alpha \gamma}:=\begin{cases}
1&\alpha = \gamma\\
0&\alpha \neq \gamma
\end{cases}
\ee 
by induction. 
It is obvious that $\phi_\alpha$ with $|\alpha|\le 1$ satisfies \eqref{eq:huinterpro}. Suppose that  \eqref{eq:huinterpro} holds for any $\phi_\alpha$ with $|\alpha|\le k$, consider $\phi_\alpha$ with $|\alpha|=k+1$. 

If $|\gamma|\le k-1$, a combination of \eqref{eq:phialpha}, \eqref{eq:Cdef} and \eqref{eq:huinterpro} gives
\begin{equation}\label{eq:alpha1} 
{1\over |K|}\int_K D^\gamma \phi_\alpha \dx 
 = { 1\over \alpha !|K|}\int_K D^\gamma (x - M_K)^\alpha \dx -  C_\alpha^\gamma =0.
\end{equation}
If $|\gamma|=k$, since $D^\gamma (x - M_K)^\alpha$ is a linear combination of $x_1-M_1$ and $x_2 - M_2$,
\begin{equation}\label{eq:alpha2} 
{1\over |K|}\int_K D^\gamma \phi_\alpha \dx 
 = { 1\over \alpha !|K|}\int_K D^\gamma (x - M_K)^\alpha \dx=0.
\end{equation}
If $|\gamma|=k+1$ and $\gamma\neq \alpha$, there must exist $i\in \{1, 2\}$ such that 
$
\gamma_i>\alpha_i,
$
which implies that
$
\partial_i^{\gamma_i} (x_i-M_i)^{\alpha_i}=0.
$
Consequently,
\be\label{eq:betagamma1}
D^\gamma \phi_\alpha = \partial_1^{\gamma_1} (x_1-M_1)^{\alpha_1}\partial_2^{\gamma_2} (x_2-M_2)^{\alpha_2}=0.
\ee
Since $\phi_\beta\in P_{|\beta|}(K, \R)$, 
\be\label{eq:betagamma2}
D^\gamma \phi_\beta = 0, \quad \mbox{ if }\quad |\gamma|>|\beta|.
\ee
If $|\gamma|=k+1$ and $\gamma\neq \alpha$, a combination of \eqref{eq:phialpha}, \eqref{eq:betagamma1} and \eqref{eq:betagamma2} gives
$\int_K D^\gamma \phi_\alpha \dx = 0$. 
If $\gamma=\alpha$, by the definition \eqref{eq:phialpha}, a direct computation yields
$
{1\over |K|}\int_K D^\gamma \phi_\alpha \dx = 1.
$ 
Since $\phi_\alpha\in P_{k+1}(K, \R)$, it is trivial that $D^\gamma \phi_\alpha=0$ for any $|\gamma|>k+1$. A combination of all the results above leads to \eqref{eq:huinterpro} for $|\alpha|=k+1$, which completes the proof
 for \eqref{eq:huinterpro}.
A combination of the definition of $\Pik_K$ in \eqref{def:pik} and \eqref{eq:huinterpro} gives \eqref{eq:pil} directly.

By \eqref{eq:pil},
$
\displaystyle \Pi_K^{l+1} u=\sum_{|\alpha|= l+1} a_K^\alpha \phi_\alpha + \sum_{|\alpha|< l+1} a_K^\alpha \phi_\alpha.
$
Since $\phi_\alpha\in P_{|\alpha|}(K,\R)$,
\begin{equation}\label{add1}
(I- \Pik_K)\Pi_K^{l+1} u=\sum_{|\alpha|= l+1} a_K^\alpha (I- \Pik_K)\phi_\alpha.
\end{equation} 
It follows from \eqref{def:pik} and \eqref{eq:huinterpro} that
\begin{equation}\label{add2}
\Pik_K \phi_\alpha = 0,\quad \forall |\alpha|=l+1.
\end{equation}
A combination of \eqref{add1} and \eqref{add2} gives 
$$
(I- \Pik_K)\Pi_K^{l+1} u=\sum_{|\alpha|= l+1} a_K^\alpha \phi_\alpha,
$$
which completes the proof for \eqref{eq:C0}.
\end{proof}

\begin{lemma}\label{Lm:m} 
It holds on uniform triangulations that
$$
\big |((I-\PiHHJ)\nabla_h^2 \Pit_h u, (I-\PiHHJ)\nabla_h^2 (I - \Pit_h)u)\big |\lesssim h^4\|u\|_{5, \Om}^2,
$$
provided $u\in H^5(\Om, \R)$.
\end{lemma}
\begin{proof}
The partition $\cT_h$ of domain $\Om$ includes the set of  parallelograms $\mathcal{N}_1$ and the set of a few remaining boundary triangles $\mathcal{N}_2$, see Fig.~\ref{fig:triangulation} for example.
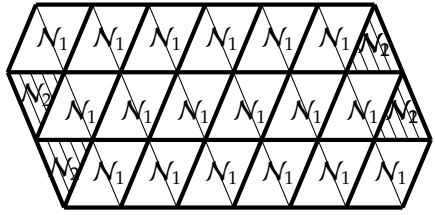
\begin{figure}[!ht]
\begin{center}
\begin{tikzpicture}[xscale=3,yscale=3]
\draw[-,ultra thick] (0.25,0) -- (1.75,0);
\draw[-,ultra thick] (0.125,0.9) -- (1.625,0.9);
\draw[-,ultra thick] (0.375,0.9) -- (0.125,0.3);
\draw[-,ultra thick] (0.625,0.9) -- (0.25,0);
\draw[-,ultra thick] (0.875,0.9) -- (0.5,0);
\draw[-,ultra thick] (1.125,0.9) -- (0.75,0);
\draw[-,ultra thick] (1.375,0.9) -- (1,0);
\draw[-,ultra thick] (1.625,0.9) -- (1.25,0);
\draw[-,ultra thick] (1.75,0.6) -- (1.5,0);
\draw[-] (0.125,0.9) -- (0.5,0);
\draw[-] (0.375,0.9) -- (0.75,0);
\draw[-] (0.625,0.9) -- (1,0);
\draw[-] (0.875,0.9) -- (1.25,0);
\draw[-] (1.125,0.9) -- (1.5,0);
\draw[-] (1.375,0.9) -- (1.75,0);
\draw[-,ultra thick] (1.625,0.9) -- (1.875,0.3);
\draw[-,ultra thick] (0,0.6) -- (0.25,0);
\draw[-,ultra thick] (0.125,0.9) -- (0,0.6);
\draw[-,ultra thick] (0,0.6) -- (1.75,0.6); 
\draw[-,ultra thick] (0.125,0.3) -- (1.875,0.3);
\draw[-,ultra thick] (1.75,0) -- (1.875,0.3);

\draw[-] (0.05,0.6) -- (0.15,0.36);
\draw[-] (0.1,0.6) -- (0.175,0.42);
\draw[-] (0.15,0.6) -- (0.2,0.48);
\draw[-] (0.2,0.6) -- (0.225,0.54);
  
\draw[-] (0.175,0.3) -- (0.275,0.06);
\draw[-] (0.225,0.3) -- (0.3,0.12);
\draw[-] (0.275,0.3) -- (0.325,0.18);
\draw[-] (0.325,0.3) -- (0.35,0.24);

\draw[-] (1.55,0.6) -- (1.525,0.66);
\draw[-] (1.6,0.6) -- (1.55,0.72);
\draw[-] (1.65,0.6) -- (1.575,0.78);
\draw[-] (1.7,0.6) -- (1.6,0.84);
 
\draw[-] (1.675,0.3) -- (1.65,0.36);
\draw[-] (1.725,0.3) -- (1.675,0.42);
\draw[-] (1.775,0.3) -- (1.7,0.48);
\draw[-] (1.825,0.3) -- (1.725,0.54);
  
\node at (0.125,0.5) {$\mathcal{N}_2$};
\node at (0.25,0.175) {$\mathcal{N}_2$};
\node at (1.625,0.725) {$\mathcal{N}_2$};
\node at (1.75,0.425) {$\mathcal{N}_2$};

\node at (0.2,0.75) {$\mathcal{N}_1$}; 
\node at (0.45,0.75) {$\mathcal{N}_1$}; 
\node at (0.7,0.75) {$\mathcal{N}_1$}; 
\node at (0.95,0.75) {$\mathcal{N}_1$}; 
\node at (1.2,0.75) {$\mathcal{N}_1$}; 
\node at (1.45,0.75) {$\mathcal{N}_1$};

\node at (0.325,0.425) {$\mathcal{N}_1$}; 
\node at (0.575,0.425) {$\mathcal{N}_1$}; 
\node at (0.825,0.425) {$\mathcal{N}_1$}; 
\node at (1.075,0.425) {$\mathcal{N}_1$}; 
\node at (1.325,0.425) {$\mathcal{N}_1$}; 
\node at (1.575,0.425) {$\mathcal{N}_1$};

\node at (0.45,0.155) {$\mathcal{N}_1$}; 
\node at (0.7,0.155) {$\mathcal{N}_1$}; 
\node at (0.95,0.155) {$\mathcal{N}_1$}; 
\node at (1.2,0.155) {$\mathcal{N}_1$}; 
\node at (1.45,0.155) {$\mathcal{N}_1$}; 
\node at (1.7,0.155) {$\mathcal{N}_1$};

\node at (1.625,0.725) {$\mathcal{N}_1$};
\node at (1.75,0.425) {$\mathcal{N}_1$};

\end{tikzpicture}
\caption{A uniform triangulation of $\Om$.}
\label{fig:triangulation}
\end{center}
\end{figure}
 Let 
 \begin{equation}\label{kappadef}
 kappa=|\mathcal{N}_2|
 \end{equation} 
 denote the number of the elements in $\mathcal{N}_2$. It holds that
\begin{equation}\label{meshtotal}
\left((I-\PiHHJ)\nabla_h^2 \Pit_h u, (I-\PiHHJ)\nabla_h^2 (I - \Pit_h )u\right)=I_{\mathcal{N}_1} + I_{\mathcal{N}_2} 
\end{equation}
with $\displaystyle I_{\mathcal{N}_i}=\sum_{K\in \mathcal{N}_i} ((I-\PiHHJ)\nabla^2 \Pit_K u, (I-\PiHHJ)\nabla^2 (I - \Pit_K )u)_{0, K}$.

It follows from the estimate \eqref{interr} 
that
\begin{equation}\label{orth1}
\begin{split}
&((I-\PiHHJ)\nabla^2 \Pit_K u, (I-\PiHHJ)\nabla^2 (I - \Pit_K)u)_{0, K}\\
= &((I-\PiHHJ)\nabla^2 \Pit_K u, (I-\PiHHJ)\nabla^2 (I - \Pit_K)\Pif_K u)_{0, K}
+\cO(h^4\|u\|_{5,\Om}^2).
\end{split}
\end{equation} 
Consider the expansion of $(I - \Pit_K)\Pif_K u$. Let $l=4$ in \eqref{eq:C0}. It holds that
\begin{equation*}\label{eq:pi3}
(I - \Pit_K)\Pif_K u \big |_K= \sum_{|\alpha|=4} a_K^\alpha \big ({1\over \alpha !} (x - M_K)^\alpha - \sum_{|\beta|\le 2} C_\alpha^\beta\phi_\beta\big ).
\end{equation*}
This implies that $(I - \Pit_K)\Pif_K u$ does not include any homogeneous third order terms. These vanishing homogeneous third order terms are  crucial for the analysis here. By the definition of the interpolation $\PiHHJ$ and the fact that  $\nabla^2 \phi_\beta$ is constant if $|\beta|=2$, 
$$
(I - \PiHHJ)\nabla^2 (I - \Pit_K )\Pif_K  u\big |_K =  \sum_{|\alpha|=4} {a_K^\alpha\over \alpha!} (I - \PiHHJ) \nabla^2 (x - M_K)^\alpha.
$$
Similarly,  
$\displaystyle 
(I - \PiHHJ)\nabla^2 \Pit_K  u|_K=  \sum_{|\beta|=3} {a_K^\beta\over \beta!} (I - \PiHHJ) \nabla^2 (x - M_K)^\beta.
$
Thus, by \eqref{orth1},
\begin{equation}\label{orth2} 
((I-\PiHHJ)\nabla^2 \Pit_K u, (I-\PiHHJ)\nabla^2 (I - \Pit_K )u)_{0, K}\\
= \sum_{|\alpha|=4}\sum_{|\beta|=3}  {a_K^\alpha a_K^\beta \over \alpha!\beta!} c_K^{\alpha\beta}
+\cO(h^4\|u\|_{5, K}^2),
\end{equation}
with 
$
c_K^{\alpha \beta} :=\big ((I - \PiHHJ) \nabla^2 (x - M_{K})^\beta, (I - \PiHHJ) \nabla^2 (x - M_{K})^\alpha\big )_{0, K}.
$

Consider $I_{\mathcal{N}_1}$ in \eqref{meshtotal}. Let $\boldsymbol{c}$ be the centroid of a parallelogram formed by two adjacent elements $K_1$ and  $K_2$. Define a mapping $T: x \rightarrow \tilde{x}=2\boldsymbol{c} - x$. It is obvious that $T$ maps $K_1$ onto $K_2$ and 
$
x - M_{K_1} = -(\tilde{x} - M_{K_2}).
$
Note that for any $|\alpha|=4$ and $|\beta|=3$, $ \nabla^2 (x - M_K)^\alpha$ and $\nabla^2 (x - M_K)^\beta$ are homogeneous polynomials of $x_1-M_1$ and $x_2-M_2$ with degree 2 and 1, respectively. For any adjacent elements $K_1$ and $K_2$ forming a parallelogram,
$
c_{K_1}^{\alpha \beta} = -c_{K_2}^{\alpha \beta}.
$
It follows that 
\begin{equation}\label{orth3}
\begin{split}
&((I-\PiHHJ)\nabla^2 \Pit_h u, (I-\PiHHJ)\nabla^2 (I - \Pit_h )u)_{0, K_1\cup K_2}\\
&= \sum_{|\alpha|=4}\sum_{|\beta|=3} {1\over \alpha!\beta!}(a_{K_1}^\beta a_{K_1}^\alpha - a_{K_2}^\beta a_{K_2}^\alpha) c_{K_1}^{\alpha \beta}
+\cO(h^4\|u\|_{5,\Om}^2)\\
&=\sum_{|\alpha|=4}\sum_{|\beta|=3} {1\over \alpha!\beta!}\big (a_{K_1}^\beta (a_{K_1}^\alpha - a_{K_2}^\alpha) +(a_{K_1}^\beta - a_{K_2}^\beta) a_{K_2}^\alpha\big ) c_{K_1}^{\alpha \beta}
+\cO(h^4\|u\|_{5,\Om}^2).
\end{split}
\end{equation} 
Recall $a_K^\alpha={1\over |K|}\int_K  D^\alpha u\dx$ in \eqref{eq:pil}. 
Note that for all $v\in P_4(K_1\cup K_2)$, we have
$
{1\over |K_1|}\int_{K_1}  D^\alpha v\dx-{1\over |K_2|}\int_{K_2}  D^\alpha v\dx =0.
$ 
The Bramble-Hilbert Lemma leads that
$$
\left|{1\over |K_1|}\int_{K_1}  D^\alpha v\dx-{1\over |K_2|}\int_{K_2}  D^\alpha v\dx\right |\lesssim  |u|_{5,K_1\cup K_2}.
$$
A similar analysis for $|\beta|=3$ gives 
\begin{equation}\label{orth4}
|a_{K_1}^\alpha - a_{K_2}^\alpha| + |a_{K_1}^\beta - a_{K_2}^\beta|\lesssim  \|u\|_{5, K_1\cup K_2}.
\end{equation} 
Note that 
\begin{equation}\label{orth5}
|c_{K_1}^{\alpha \beta}|\lesssim h^2 \|\nabla^3 (x - M_{K_1})^\beta\|_{0,K_1} \|\nabla^3 (x - M_{K_1})^\alpha\|_{0,K_1} \lesssim h^3|K_1|\quad \mbox{for }\quad |\alpha|=4,\ |\beta|=3.
\end{equation} 
A substitution of \eqref{orth4} and \eqref{orth5} into \eqref{orth3} gives 
\be\label{meshtotal1}
\begin{split}
I_{\mathcal{N}_1}&\lesssim \sum_{K\in\mathcal{N}_1} h^4 \|u\|_{5, K}|K|^{\frac12}+\cO(h^4\|u\|_{5,\Om}^2)\\
& \le h^4 \left(\sum_{K\in\mathcal{N}_1} \|u\|_{5, K}^2\right)^{\frac12}\left(\sum_{K\in\mathcal{N}_1} |K|\right)^{\frac12}+\cO(h^4\|u\|_{5,\Om}^2)
\lesssim h^4\|u\|_{5, \Om}^2.
\end{split}
\ee
For any element $K\in \mathcal{N}_2$, it follows from \eqref{orth2}  and \eqref{orth5} that
$$
\big | ((I-\PiHHJ)\nabla^2 \Pit_K u, (I-\PiHHJ)\nabla^2 (I - \Pit_K )u)_{0, K} \big |\lesssim h^3  |K|.
$$
Since $\sum_{K\in \mathcal{N}_2}|K|\lesssim h$, it follows that
\be\label{meshtotal2}
I_{\mathcal{N}_2}\lesssim h^4\|u\|_{5, \Om}^2.
\ee
A substitution of \eqref{meshtotal1} and \eqref{meshtotal2} into \eqref{meshtotal} leads to
$$
\big |((I-\PiHHJ)\nabla^2 \Pit_h u, (I-\PiHHJ)\nabla^2 (I - \Pit_h )u) \big | \lesssim h^4\|u\|_{5, \Om}^2,
$$
which completes the proof.
\end{proof}

Thanks to Lemmas \ref{lm:gammaconstant}, \ref{RT2h40} and  \ref{Lm:m}, there exists the following fourth-order accurate expansion of $\parallel (I-\PiHHJ)\nabla^2 u\parallel_{0, \Om}^2$ in Theorem \ref{Th:extraCR}.
\begin{lemma}\label{RT2h4}
For any  $u\in V\cap H^{5}(\Om,\mathbb{R})$,
\begin{equation} \label{identity:RTu} 
\parallel (I-\PiHHJ)\nabla^2 u\parallel_{0,\Om}^2= {1\over N}F(u, \Om)|\Omega| + \cO(h^4\|u\|_{5, \Om}^2),
\end{equation}
where $F(u, \Om)$ in \eqref{def:F} is independent of the mesh size $h$ and N is the number of elements.
\end{lemma}

\subsection{Error estimate of $I_1=(\sigma - \sHHJc, \sHHJc - \sHHJ)$}
The first order term $\sigma - \sHHJc$ leads to a suboptimal result of $I_1$ if 
the Cauchy-Schwarz inequality is applied directly.
To prove an optimal  error estimate, we employ the commuting property of the interpolation of the Morley element and the equivalence between the HHJ element and the  modified Morley element. This allows to express $I_1$ in terms of  the error of $\uHHJ$ with a second order accuracy, which leads to the desired optimal estimate.
\begin{lemma}\label{lm:I1}
Suppose that $(\lambda, u)$ is the solution  of  Problem \eqref{variance} and $u\in V\cap H^3(\Om, \R)$. It holds that 
$$
|I_1|\lesssim h^4\|u\|_{3,\Om}^2.
$$
\end{lemma}

\begin{proof}
By the equivalence \eqref{relation} between the HHJ element and  modified  Morley element,
\be\label{eq:I11}
(\sigma - \sHHJc, \sHHJc - \sHHJ)  = (\nabla_h^2 (u - \Mumc), \nabla_h^2 ( \Mumc - \Mum)). 
\ee
Since $\Mumc - \Mum\in \VM$, the commuting property \eqref{commuting} of the Morley element  leads to 
\be\label{eq:I12}
(\nabla_h^2 (u - \Mumc), \nabla_h^2 ( \Mumc - \Mum)) = (\nabla_h^2 (\PiM u - \Mumc), \nabla_h^2 ( \Mumc - \Mum)).
\ee
Recall that $\Mumc$ and $\Mum$ are the solutions of Problem \eqref{eq:Mmc} with $f=\lambda u$ and $\Mlam\Mu$, respectively. It follows that
\be\label{eq:I13}
(\nabla_h^2 (\PiM u - \Mumc), \nabla_h^2 ( \Mumc - \Mum)) = (\lambda u - \Mlam\Mu, \Pi_{\rm D} (\PiM u - \Mumc)).
\ee
Thanks to the error estimate \eqref{M:est} of the Morley element and \eqref{RT:est} of the HHJ element, the equivalence \eqref{relation} between the  modified Morley element and the HHJ element,  and the triangle inequality,  
\be\label{eq:I14}
\|\Pi_{\rm D} (\PiM u - \Mumc)\|_{0, \Om} \le \|\Pi_{\rm D} (\PiM u - u) \|_{0, \Om} + \|u- \uHHJc\|_{0, \Om} \le h^2\|u\|_{3,\Om}.
\ee
A combination of  \eqref{M:est}, \eqref{eq:I11}, \eqref{eq:I12}, \eqref{eq:I13} and \eqref{eq:I14} gives
\begin{equation*}
|(\sigma - \sHHJc, \sHHJc - \sHHJ)|\le h^2(\lambda |u-\Mu| + |\lambda - \Mlam|)\|u\|_{3,\Om}\le h^4\|u\|_{3,\Om}^2,
\end{equation*}
which completes the proof. 

\end{proof}

\subsection{Error estimate of $I_2=(\sigma - \sHHJ, \nabla_h^2 (\Mum-  \Mu))$}
This term is essentially a consistency error term of the Morley element with third order accuracy by a direct use of the Cauchy-Schwarz inequality. The main idea here is to employ the equivalence between the HHJ element and the  modified Morley element and make use of the weak continuity of solutions in $\VM$.  

\begin{lemma}\label{lm:I2}
Suppose that $(\lambda, u)$ is the solution  of Problem  \eqref{variance} and $u\in V\cap H^{3}(\Om, \R)$. It holds 
 that
$$
|I_2|\lesssim h^4\|u\|_{3,\Om}^2.
$$
\end{lemma}
\begin{proof}
By the equivalence \eqref{relation} between the HHJ element and the modified  Morley element and the superclose property in Lemma \ref{supereig1},
\begin{equation*}
\begin{split}
I_2=&(\nabla^2 (u -  \Mum), \nabla_h^2 (\Mum-  \Mu))
=(\nabla^2 (u -  \Mu), \nabla_h^2 (\Mum-  \Mu)) + \cO(h^4\|u\|_{3, \Om}^2).
\end{split}
\end{equation*}
By the commuting property \eqref{commuting} of the Morley element and  Problems \eqref{discrete}, \eqref{eq:Mmc},
\begin{equation}\label{en} 
(\nabla^2 (u -  \Mu), \nabla_h^2 (\Mum-  \Mu))=\Mlam(\Mu, (I - \Pi_{\rm D})s_h)
\end{equation}
with $s_h=\PiM u -  \Mu$. Note that
\begin{equation}\label{en1}
\begin{split}
(\Mlam\Mu, (I - \Pi_{\rm D})s_h)=(\lambda \Pi_h^0 u, (I - \Pi_{\rm D})s_h) + (\Mlam\Mu - \lambda \Pi_h^0 u, (I - \Pi_{\rm D})s_h).
\end{split}
\end{equation}
It follows from the triangle inequality and the error estimate \eqref{M:est} that
\begin{equation}\label{en2}
\begin{split}
\left|(\Mlam\Mu - \lambda \Pi_h^0 u, (I - \Pi_{\rm D})s_h)\right|\lesssim h^4\|u\|_{3,\Om}^2.
\end{split}
\end{equation}
According to the expansion of the interpolation error in \cite{huang2008superconvergence} for the linear element,
\begin{equation}\label{en3} 
(\lambda \Pi_h^0 u, (I - \Pi_{\rm D})s_h)= (-\frac12 \sum_{i=1}^3{\partial^2 s_h\over \partial \bold{t}_i^2}\psi_{i-1}\psi_{i+1}|e_i|^2, \lambda \Pi_h^0 u)
=-{\lambda \over 24}\sum_{i=1}^3 |e_i|^2( {\partial^2 s_h\over \partial \bold{t}_i^2}, \Pi_h^0u),
\end{equation}
where $\psi_i$ are the barycentric coordinates. Integration by parts indicates that
\begin{equation}\label{en4}
\begin{split}
\int_K {\partial^2 s_h\over \partial \bold{t}_i^2} \Pi_h^0u\dx =&\int_{\partial K} {\partial s_h\over \partial \bold{t}_i} \Pi_h^0u(\bold{n}\cdot \bold{t}_i)\ds.
\end{split}
\end{equation}
As $s_h\in \VM$, it holds that
$
\int_e [\nabla s_h]\ds=0.
$
Note that $\bold{t}_e$ is the unit tangent vector of edge $e$ with the same direction on two adjacent elements sharing the edge $e$, and $\bold{t}_i$ is the one along the boundary of an element with counterclockwise orientation.
Since $\bold{n}\cdot \bold{t}_i$ is the same on two adjacent elements, it follows \eqref{M:est} that 
\begin{equation}\label{en6}
\begin{split}
\left| \sum_{K\in\cT_h} |e_i|^2\int_{\partial K} {\partial s_h\over \partial \bold{t}_i} \Pi_h^0u(\bold{n}\cdot \bold{t}_i)\ds\right|
&=\left|\sum_{e\in\cE_h}(\bold{n}\cdot \bold{t}_i) |e_i|^2\int_{e} {\partial s_h\over \partial \bold{t}_i} [(\Pi_h^0-\Pi_e^0)u]\ds\right|
\\
&\lesssim h^2|s_h|_{1,h}\|u\|_{1,\Om}\lesssim h^4\|u\|_{3,\Om}^2. 
\end{split}
\end{equation}
By the commuting property of the interpolation $\PiM$ in \eqref{commuting}, a substitution of \eqref{en1}, \eqref{en2}, \eqref{en3}, \eqref{en4} and \eqref{en6} into \eqref{en} gives
$$
|I_2|\lesssim h^4\|u\|_{3,\Om}^2,
$$
which completes the proof.
\end{proof}

\subsection{Error estimate of $I_3=\lambda (u-\PiM u, u )$} 
This interpolation term $I_3$ can not be cancelled with other terms as the analysis in \cite{hu2019asymptotic}  for the Crouzeix-Raviart element.
The key here to obtain an optimal estimate is to exploit the Taylor expansion of $(I-\PiM)u$ and make full use of the uniform triangulations. 

Define
\be\label{Mbase}
\phiM^i = \psi_{i-1}^2\psi_{i+1},\quad \phiM^4=\psi_1\psi_2\psi_3,\quad 1\le i\le 3
\ee
with the barycentric coordinates $\{\psi_i\}_{i=1}^3$. For two adjacent elements $K_1$ and $K_2$ forming a parallelogram, let the local index of vertices satisfy $\nabla \psi_{i}|_{K_1}=-\nabla \psi_{i}|_{K_2},~i=1,2,3.$
By \eqref{nlambda}, there exist the following properties of these cubic polynomials
\begin{equation}\label{eq:Mproperty}
\begin{split}
{\partial^3\phiM^i\over \partial\bt_j^3} =-{2\over |e_i|^3} \delta_{ij}, \quad {\partial^3\phiM^4\over \partial\bt_{j-1}^2\partial \bt_{j+1}}={2\over |e_{j+1}||e_{j-1}|^2}, \quad 1\le i\le 4,\ 1\le j\le 3,
\\
{\partial^3\phiM^i\over \partial\bt_{j-1}^2\partial \bt_{j+1}}=\begin{cases}
-{2\over |e_{j+1}||e_{j-1}|^2} &i=j+1\\
0&i=j\\
{2\over |e_{j+1}||e_{j-1}|^2} &i=j-1
\end{cases},\quad 1\le i, j\le 3.
\end{split}
\end{equation}
Note that 
\be\label{p32}
P_3(K, \R)=P_2(K, \R) + \mbox{span}\{\phiM^i,\ 1\le i\le 4\}.
\ee
For $1\le i\le 3$, $\psi_{i}^{2} \psi_{i-1}=\phiM^{i+1}$ and 
$$
\psi_{i}^{3}=-\phiM^{i+1}+\phiM^{i-1}+\phiM^{4}-\psi_{i} \psi_{i+1}+\psi_{i}^{2},\quad
\psi_{i}^{2} \psi_{i+1}=-\phiM^{i-1}-\phiM^{4}+\psi_{i} \psi_{i+1}.
$$ 


\begin{lemma}\label{lm:morleyexpan}
For any $w\in P_3(K,\R)$, 
\be \label{eq:I31}
(I-\PiM )\Pit_h w = \sum_{i=1}^4 c_i(I-\PiM )\phiM^i
\ee
with 
\begin{equation} \label{eq:I30}
\begin{split}
c_j &= -{|e_j|^3\over 2} \Pi_h^0 {\partial^3 w\over \partial\bt_j^3},\qquad 1\le j\le 3,\\  
c_4 &= {|e_{j+1}||e_{j-1}|^2\over 2} \Pi_h^0 {\partial^3 w\over \partial\bt_{j-1}^2\partial \bt_{j+1}} -{|e_{j+1}|^3\over 2}\Pi_h^0{\partial^3 w\over \partial\bt_{j+1}^3} + {|e_{j-1}|^3\over 2}\Pi_h^0 {\partial^3 w\over \partial\bt_{j-1}^3}.
\end{split}
\end{equation}
\end{lemma}
\begin{proof}
By \eqref{p32}, there exist constants $c_i$ and $p_2\in P_2(K, \R)$ such that
\be\label{eq:I300}
\Pit_h w = \sum_{i=1}^4 c_i\phiM^i + p_2.
\ee
It leads to \eqref{eq:I31} 
with coefficients $c_i$ to be determined. Taking third derivatives on both sides of \eqref{eq:I300}, it follows from \eqref{eq:Mproperty} that
\ben
{\partial^3\Pit_h w\over \partial\bt_j^3} = -{2\over |e_j|^3}c_j, \qquad {\partial^3\Pit_h w\over \partial\bt_{j-1}^2\partial \bt_{j+1}} = {2\over |e_{j+1}||e_{j-1}|^2}(c_{j-1} - c_{j+1} + c_4).
\een
The fact that $\partial^\alpha \Pit_h w=\Pi_h^0 \partial^\alpha w$ with $|\alpha|=3$ leads to  \eqref{eq:I30} directly, and completes the proof.
\end{proof}

The  expansion in Lemma \ref{lm:morleyexpan} of the interpolation error of the Morley  leads to the following optimal analysis of $I_3$ on uniform triangulations.
\begin{lemma}\label{lm:I3}
Suppose that $(\lambda, u)$ is the solution of  Problem \eqref{variance} and $u\in V\cap H^{4}(\Om, \R)$. It holds on uniform triangulations that
$$
|I_3|\lesssim h^4 \|u\|_{4,\Om}^2.
$$
\end{lemma}
\begin{proof}
Note that 
\begin{equation}\label{eq:I32} 
I_3=\lambda ((I-\PiM)\Pit_h u, \Pi_h^0 u) + \lambda ((I-\PiM)(I-\Pit_h)u, \Pi_h^0 u) + \lambda ((I-\PiM)u, (I - \Pi_h^0)u),
\end{equation} 
\be\label{eq:I33}
\big |((I-\PiM)(I-\Pit_h)u, \Pi_h^0u)\big | + \big |((I-\PiM)u, (I - \Pi_h^0)u)\big |\lesssim h^4\|u\|_{4,\Om}^2.
\ee
It follows from \eqref{eq:I31}, \eqref{eq:I32} and \eqref{eq:I33} that
\begin{equation}\label{eq:I34}
\begin{split}
I_3=\lambda ((I-\PiM)\Pit_h u, \Pi_h^0 u) + \cO(h^4|u|_{4,\Om}^2)
=\lambda \sum_{K\in\mathcal{T}_h}\sum_{i=1}^4s_i (c_i , \Pi_K^0 u)_K  + \cO(h^4|u|_{4,\Om}^2).
\end{split}
\end{equation}
with $s_i= \Pi_K^0(I-\PiM)\phiM^i $ and $c_i$ defined in \eqref{eq:I30}. 
By the definition of $\phiM^i$ in \eqref{Mbase}, constant $s_i$ has the same value on different elements. 
Note that the directions of $ \bt_{i}$ for two adjacent elements are opposite. Thus, the definition of $c_i$ in \eqref{eq:I30} indicates 
$$
c_i|_{K_1} = -c_i|_{K_2} = \cO(h^3),
$$
where adjacent elements $K_1$ and $K_2$ forming a parallelogram. Similar to the analysis in Lemma \ref{Lm:m},
$
\big |\sum_{i=1}^4s_i (c_i , \Pi_h^0 u)\big |\lesssim h^4\|u\|_{3, \Om}^2.
$
A substitution of this estimate into \eqref{eq:I34} leads to
\ben
|I_3|\lesssim h^4\|u\|_{4,\Om}^2,
\een
which completes the proof.
\end{proof}

Lemma \ref{lm:I1},  \ref{lm:I2},   \ref{lm:I3} show that the terms $I_1$, $I_2$ and $I_3$ are higher order terms. According to the decomposition of eigenvalues in Theorem \ref{Th:extraCR}, the error of eigenvalues equals to the interpolation error of the HHJ element in $L^2$ norm up to a higher order term. The following asymptotic expansion of eigenvalues of the Morley element comes from the combination of Lemma \ref{RT2h4}, \ref{lm:I1},  \ref{lm:I2},   \ref{lm:I3}  and Theorem \ref{Th:extraCR}. 

\begin{Th}\label{M:extra}
Suppose that $(\lambda , u )$ is the eigenpair of Problem \eqref{variance} with $u \in V\cap H^{5}(\Om,\mathbb{R})$, and $(\Mlam, \Mu)$ is the corresponding  eigenpair of Problem \eqref{discrete} by the Morley element on an uniform triangulation $\cT_h$. It holds that
\begin{equation*}
\begin{split}
\lambda-\Mlam= \frac1N F(u, \Om)|\Om| +\cO(h^4|\ln h|\|u\|_{5,\Om}^2).
\end{split}
\end{equation*}
with $N$ the number of elements on $\mathcal{T}_h$ and $F(u, \Om)$ defined in \eqref{def:F}.
\end{Th}

The optimal convergence result of the extrapolation algorithm in the following theorem is  an immediate consequence of Theorem \ref{M:extra}. 
\begin{Th}\label{M}
Suppose that $\lambda$ is a simple eigenvalue of Problem \eqref{variance}
and a corresponding eigenfunction $u \in V\cap H^{5}(\Om,\mathbb{R})$. Let $(\Mlam, \Mu)$ be the corresponding  eigenpair of Problem \eqref{discrete} by the Morley element on an uniform triangulation $\cT_h$. It holds that
\begin{equation*}
\begin{split}
\big |\lambda - \lambda_{\rm M}^{\rm EXP}\big |\lesssim h^4|\ln h| \|u \|_{5, \Omega}^2.
\end{split}
\end{equation*}
with extrapolation eigenvalues $\lambda_{\rm M}^{\rm EXP}$ defined in \eqref{exmethod}.
\end{Th}

\begin{remark}
 If $\lambda$ is a multiple eigenvalue, eigenfunctions $\Mu$ on triangulations with different mesh size may approximate to different functions $u\in M(\lambda)$. Then, the asymptotic expansion of eigenvalue $\Mlam$ in Theorem \ref{M:extra} cannot lead to a theoretical estimate of $\lambda_{\rm M}^{\rm EXP}$ in \eqref{exmethod}. Some numerical tests in Section \ref{sec:numerical} show that extrapolation method can also  improve the accuracy of multiple eigenvalues  to $\mathcal{O}(h^4)$ if the eigenfunction is smooth enough.
\end{remark}

\section{Asymptotically exact a posterior error estimate and postprocessing scheme}\label{sec:post}

In this section, we establish and analyze an asymptotically exact a posterior error estimate of eigenvalues by the Morley element. This estimate gives a high accuracy postprocessing scheme of eigenvalues.

To begin with, we introduce a gradient recovery technique \cite{brandts1994superconvergence,Hu2016Superconvergence}.
Define the discrete space
\begin{equation*}
{\rm CR }(\mathcal{T}_h, \S):=\big \{v\in L^2(\Om,\S): v|_K\in P_1(K, \S)\text{ for any }  K\in\cT_h, \int_e [v]\ds =0\text{ for any }  e\in \cE_h^i\big \}.
\end{equation*}
For any function $v\in {\rm CR }(\mathcal{T}_h, \S)$, it is known that each entry of $v$ is uniquely determined by its value at the midpoint of edges, so does the function $v$ itself. Given $\textbf{q}~\in~\SHHJ$, define $K_h \textbf{q}|_K\in {\rm CR }(\mathcal{T}_h, \S)$ as follows.

\begin{Def}\label{Def:R}
1.For each interior edge $e\in\cE_h^i$, the elements $K_e^1$ and $K_e^2$ are the pair of elements sharing $e$. Then the value of $K_h \textbf{q}$ at the midpoint $\textbf{m}_e$ of $e$ is
$$
K_h \textbf{q}(\textbf{m}_e)={1\over 2}(\textbf{q}|_{K_e^1}(\textbf{m}_e)+\textbf{q}|_{K_e^2}(\textbf{m}_e)).
$$

2.For each boundary edge $e\in\cE_h^b$, let $K$ be the element having $e$ as an edge, and $K'$ be an element sharing an edge $e'\in\cE_h^i$ with $K$. Let $e''$ denote the edge of $K'$ that does not intersect with $e$, and $\textbf{m}$, $\textbf{m}'$ and $\textbf{m}''$ be the midpoints of the edges $e$, $e'$ and $e''$, respectively. Then the value of $K_h \textbf{q}$ at the point $\textbf{m}$ is
$$
K_h \textbf{q}(\textbf{m})=2K_h \textbf{q}(\textbf{m}') - K_h \textbf{q}(\textbf{m}'').
$$
\begin{center}
\begin{tikzpicture}[xscale=2,yscale=2]
\draw[-] (-0.5,0) -- (2,0);
\draw[-] (0,0) -- (0.5,1);
\draw[-] (0.5,1) -- (2,1);
\draw[-] (0.5,1) -- (1.5,0);
\draw[-] (2,1) -- (1.5,0);
\node[below, right] at (1,0.5) {\textbf{m}'};
\node[above] at (1.25,1) {\textbf{m}''};
\node[below] at (0.75,0) {\textbf{m}};
\node at (0.7,0.4) {K};
\node at (1.4,0.75) {K'};
\node at (1,0) {e};
\node at (1.4,0.2) {e'};
\node at (1.7,1) {e''};
\node at (0.3,-0.1) {$\partial \Om $};
\fill(1,0.5) circle(0.5pt);
\fill(1.25,1) circle(0.5pt);
\fill(0.75,0) circle(0.5pt);
\end{tikzpicture}
\end{center}
\end{Def}


The Morley element solution for source problems admits a first order superconvergence on uniform triangulations \cite{hu2021optimal}. According to Lemma \ref{supereig1}, the eigenfunction $ \Mu$ is superclose to  the Morley element solution for a corresponding source problem.
These two facts lead to the following superconvergence result on uniform triangulations
\begin{equation}\label{CRsuper}
\parallel \nabla^2 u - K_h \nabla_h^2 \Mu \parallel_{0,\Om}\lesssim h^2|\ln h|^{1/2} |u|_{\frac{9}{2},\Om},
\end{equation}
provided that $u\in V\cap H^{\frac{9}{2}}(\Om,\mathbb{R})$.  


Define the following asymptotically exact a posteriori error estimate
\begin{equation}\label{FdefineCR}
F_{\rm M}=\parallel K_h\nabla_h^2 \Mu  -\nabla_h^2 \Mu\parallel_{0,\Om}^2.
\end{equation}
Given a discrete eigenvalue $\Mlam$ of the Morley element, the postprocessing scheme computes a new approximate eigenvalue
\begin{equation}\label{lam:F}
\lambda_{\rm M}^{\rm R}=\Mlam + F_{\rm M}.
\end{equation} 
By the expansion \eqref{commutId}, a similar analysis to the one in \cite{hu2019asymptotic} leads to the following theorem.
\begin{Th}\label{Th:m}
Let $(\lambda , u )$ be an eigenpair of \eqref{variance} with $u \in V\cap H^{\frac{9}{2}}(\Om,\mathbb{R})$, and $(\Mlam,\Mu)$ be the corresponding approximate eigenpair of \eqref{discrete} in $\VM$. The new eigenvalue $\lambda_{\rm M}^{\rm R}$ by the postprocessing scheme satisfies
$$
|\lambda - \lambda_{\rm M}^{\rm R}| 
\lesssim h^3|\ln h| |u |_{{9\over 2}, \Omega}^2,
$$
and the accuracy of the asymptotically exact a posterior error estimates $F_M$ is $\mathcal{O}(h^3|\ln h|)$.
\end{Th}
\begin{proof}
Recall the expansion \eqref{eq:deco0} of eigenvalues
$$
\lambda-\Mlam =\|\nabla_h^2 (u - \Mu)\|_{0, \Om}^2-2\lambda (u-\PiM u, u ) + \cO(h^4|u|_{3, \Om}^2). 
$$ 
By the definition of $\lambda_{\rm M}^{\rm R}$ in \eqref{lam:F},
\begin{equation}\label{lamFerr}
\lambda - \lambda_{\rm M}^{\rm R} = \|\nabla_h^2 (u - \Mu)\|_{0, \Om}^2  - \|K_h \nabla_h^2 \Mu - \nabla_h^2 \Mu\|_{0, \Om}^2-2\lambda (u-\PiM u, u ) + \cO(h^4|u|_{3, \Om}^2).
\end{equation}
It follows from \eqref{CRsuper} and the Cauchy Schwarz inequality that
\begin{equation}\label{F1est}
\left |\|\nabla_h^2 (u - \Mu)\|_{0, \Om}^2 - \|K_h \nabla_h^2 \Mu - \nabla_h^2 \Mu\|_{0, \Om}^2\right|\lesssim h^3|\ln h|^{1/2} |u|_{\frac{9}{2},\Om}.
\end{equation}
It follows \eqref{M:est} that $|(u-\PiM u, u )|\lesssim h^3\|u\|_{3,\Om}$. A substitution of this estimate and \eqref{F1est} into \eqref{lamFerr} yields 
$$
|\lambda - \lambda_{\rm M}^{\rm R}|\lesssim h^3|\ln h| |u |_{{9\over 2}, \Omega}^2,
$$
which completes the proof.
\end{proof}

\section{Numerical examples}\label{sec:numerical}
This section presents some numerical tests to confirm the theoretical analysis in the previous sections. 
\subsection{Example 1 (simply support plate and clamped plate)}
We consider the plate problem \eqref{variance} on the unit square $\Om=(0,1)^2$.  The initial mesh $\cT_1$ consists of two right triangles, obtained by cutting the unit square with a north-east line. Each mesh $\cT_i$ is refined into a half-sized mesh uniformly, to get a higher level mesh $\cT_{i+1}$.
 
\paragraph{Simply support plate}
Consider the simply support plate with boundary condition $u={\partial^2 u\over \partial n^2} =0$.
It is known that the first eigenvalue of this problem is $\lambda_1 =4\pi^4$, and the convergence rate of approximate eigenvalues by the Morley element is $\alpha=2$. Fig.~\ref{fig:squareMorley} plots the errors of eigenvalues $\Mlam$,  $\lambda_{\rm M}^{\rm EXP}$ and $\lambda_{\rm M}^{\rm R}$ by the Morley element, the corresponding extrapolation method and the postprocessing technique, respectively. The convergence rate of eigenvalues by the Morley element is improved remarkably from second order to fourth order, which verifies the analysis in Theorem \ref{M} and Theorem \ref{Th:m}. Fig.~\ref{fig:squareMorley} also shows that the postprocessing technique has a better performance than the extrapolation method, although the theoretical convergence rate is smaller than the extrapolation method.

\begin{figure}[!ht]
\setlength{\abovecaptionskip}{0pt}
\setlength{\belowcaptionskip}{0pt}
\centering
\includegraphics[width=10cm,height=8cm]{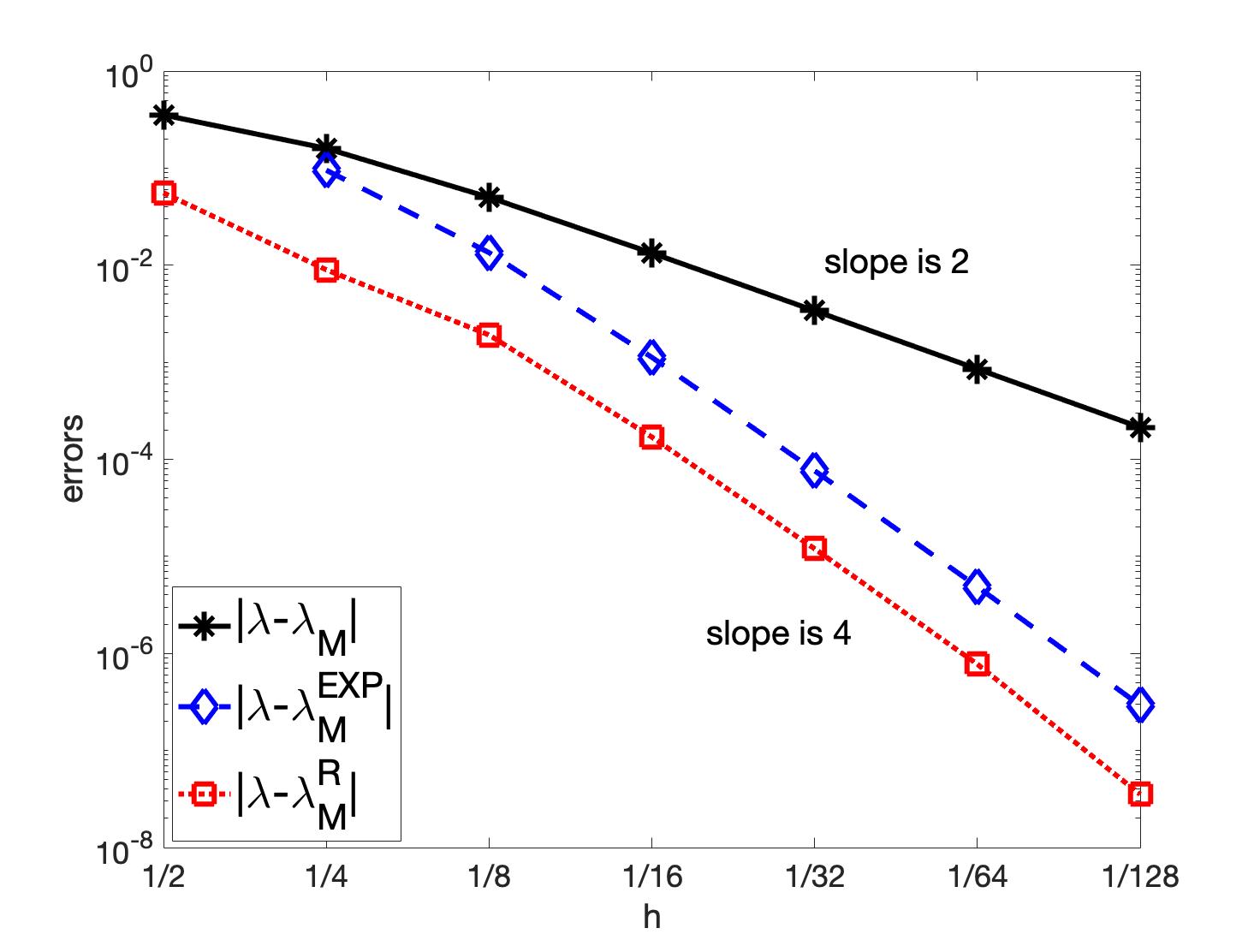}
\caption{\footnotesize{The relative errors of  the extrapolation eigenvalues  for the simply support plate in Example 1.}}
\label{fig:squareMorley}
\end{figure}

Among the smallest six eigenvalues, it is known that $\lambda_2=\lambda_3$ and $\lambda_5=\lambda_6$. 
Table~\ref{tab:multiple} lists the relative errors of eigenvalues by the Morley element for these multiple eigenvalues.  
It shows that the extrapolation method also improves the convergence rate of multiple eigenvalues to a rate of 4.
\renewcommand\arraystretch{1.5}
\begin{table}[ht!]
  \centering 
    \begin{tabular}{c|c|c|c|c|c|c|c}
\hline
          &   $\cT_2$    &  $\cT_3$     &   $\cT_4$    &    $\cT_5$   &   $\cT_6$    &  $\cT_7$     &  rate \\\hline
$\lambda_2$    & 2.55E-01 & 6.95E-02 & 8.39E-03 & 6.54E-04 & 4.35E-05 & 2.76E-06& 3.98 \\\hline
$\lambda_3$   & 2.34E-01 & 6.05E-02 & 7.20E-03 & 5.55E-04 & 3.68E-05 & 2.33E-06 & 3.98 \\\hline
$\lambda_5$    & 4.63E-01 & 1.54E-01 & 2.81E-02 & 2.64E-03 & 1.87E-04 & 1.21E-05 &  3.95\\\hline
$\lambda_6$    & 4.63E-01 & 1.55E-01 & 2.82E-02 & 2.65E-03 & 1.87E-04 & 1.21E-05 &  3.95\\\hline
    \end{tabular}%
\caption{\footnotesize The relative errors of eigenvalues by the extrapolation method for the Morley element in Example 1.}
  \label{tab:multiple}
\end{table}%
 
\begin{figure}[!ht]
\setlength{\abovecaptionskip}{0pt}
\setlength{\belowcaptionskip}{0pt}
\centering
\includegraphics[width=10cm,height=8cm]{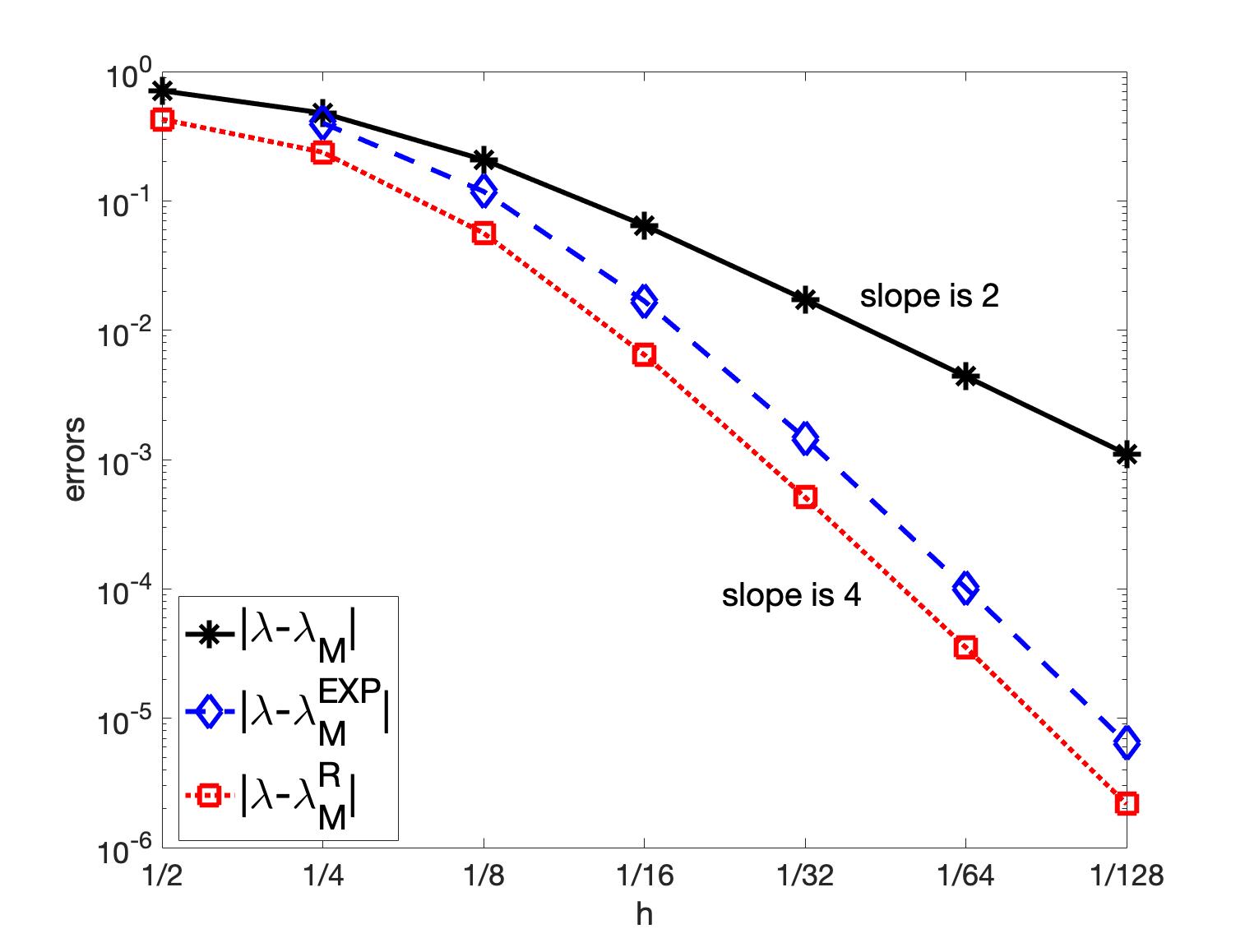}
\caption{\footnotesize{The relative errors of  the extrapolation eigenvalues for clamped plate in Example 1.}}
\label{fig:clamped}
\end{figure}

\paragraph{Clamped plate}
Consider the clamped plate with boundary condition $u={\partial u\over \partial n}=~0$.
The sum of the eigenvalue by the Morley element on the mesh $\cT_{11}$ and the corresponding a posteriori error estimate in \eqref{FdefineCR} is taken as the reference eigenvalue for this example. Fig.~\ref{fig:clamped} plots the relative errors of the first eigenvalues  $\Mlam$,  $\lambda_{\rm M}^{\rm EXP}$ and $\lambda_{\rm M}^{\rm R}$. It shows that both the extrapolation method and the postprocessing technique can improve the accuracy of eigenvalues on the clamped plate to order 4, which verifies the analysis in Theorem \ref{M} and Theorem \ref{Th:m}. Similar to the results in Example 1 for simply supported plate, Fig.~\ref{fig:clamped} implies that the accuracy of eigenvalues by the postprocessing technique is better than that of eigenvalues by the extrapolation method.


\subsection{Example 2 (Piecewise uniform mesh)}
Consider a model problem \eqref{variance}  with $V= H_0^2(\Om)$ on the domain $\Om$ shown in Fig.~\ref{fig:five}, where the coordinates of $A_1$ to $A_5$ are
$(2,0)$, $(1,-2)$, $(-1,-2)$, $(-2,0)$, $(0,2)$, respectively. The minimal and maximal angle of $\Omega$ are $90^{\circ}$ and  $121.3^{\circ}$, respectively. Fig.~\ref{fig:five} plots the initial mesh $\mathcal{T}_1$ and each mesh $\cT_i$ is refined into a half-sized mesh uniformly, to get a higher level mesh $\cT_{i+1}$.
The eigenvalue by the Aygris element on $\cT_8$ is taken as the reference eigenvalue for this example. 

\begin{figure}[H] \setlength\unitlength{1pt}
	\begin{center}
		\begin{tikzpicture}
		\draw (0,2)--(2,0);
		\draw  (2,0)--(1,-2);
		\draw  (1,-2)--(-1,-2);
		\draw  (-1,-2)--(-2,0);
		\draw  (-2,0)--(0,2);
		\draw (0,0)--(2,0);
		\draw  (0,0)--(1,-2);
		\draw  (0,0)--(-1,-2);
		\draw  (0,0)--(-2,0);
		\draw  (0,0)--(0,2);
		
		\draw[fill] (0,0) circle [radius=0.05];
		\node[below] at (0,0)  {$A_0$};
		\draw[fill] (2,0) circle [radius=0.05];
		\node[right] at (2,0)  {$A_1$};
		\draw[fill] (1,-2) circle [radius=0.05];
		\node[below] at (1,-2)  {$A_2$};
		\draw[fill] (-1,-2) circle [radius=0.05];
		\node[below] at (-1,-2)  {$A_3$};
		\draw[fill] (-2,0) circle [radius=0.05];
		\node[left] at (-2,0)  {$A_4$};
		\draw[fill] (0,2) circle [radius=0.05];
		\node[above] at (0,2)  {$A_5$};
		\end{tikzpicture}
		\label{fig:five}
		\caption{The computational domain of Example 2.}
	\end{center}
\end{figure}
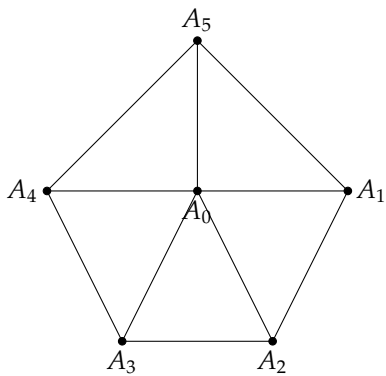

Table \ref{tab:5} indicates an almost fourth order convergence rate of extrapolation eigenvalues $\Mlam^{\rm EXP}$ in \eqref{exmethod} with $\alpha=2$  and postprocessed eigenvalues $\Mlam^{\rm R}$ in \eqref{FdefineCR}. 
Although the meshes are no longer uniform, we can still observe the optimal estimate for the asymptotic expansion of the Morley element on such piecewise uniform meshes. 
It implies that both the extrapolation method and the postprocessing technique are effective on piecewise uniform meshes. Similar to Example~1, the accuracy of eigenvalues by the postprocessing technique is slightly better than that of eigenvalues by the extrapolation method.   

\begin{table}[htbp]
	\footnotesize
	\centering 
	\begin{tabular}{c|c|c|c|c|c|c}
		\hline
		&    $\cT_3$ &  $\cT_4$ &  $\cT_5$ &  $\cT_6$ &  $\cT_7$ &   $order$  \\\hline
		$\lambda_{\rm M}$  & 9.91E-01 & 2.71E-01 & 6.95E-02 & 1.75E-02 &4.38E-03 &1.99 \\\hline
		$\lambda_{\rm M}^{\rm R}$  &1.51E-01 & 1.37E-02 & 1.05E-03 & 7.95E-05 & 6.38E-06&3.64  \\\hline
		$\lambda_{\rm M}^{\rm EXP}$ & 3.21E-01 & 3.11E-02 & 2.28E-03 & 1.52E-04 & 9.81E-06 & 3.95 \\\hline
	\end{tabular}
	\caption{The relative errors of eigenvalues for Example 2.}
	\label{tab:5}%
\end{table}%

\subsection{Example 3 (Cracked domain)}
\begin{figure}[!ht]
\setlength{\abovecaptionskip}{0pt}
\setlength{\belowcaptionskip}{0pt}
\centering
\includegraphics[width=5cm,height=4cm]{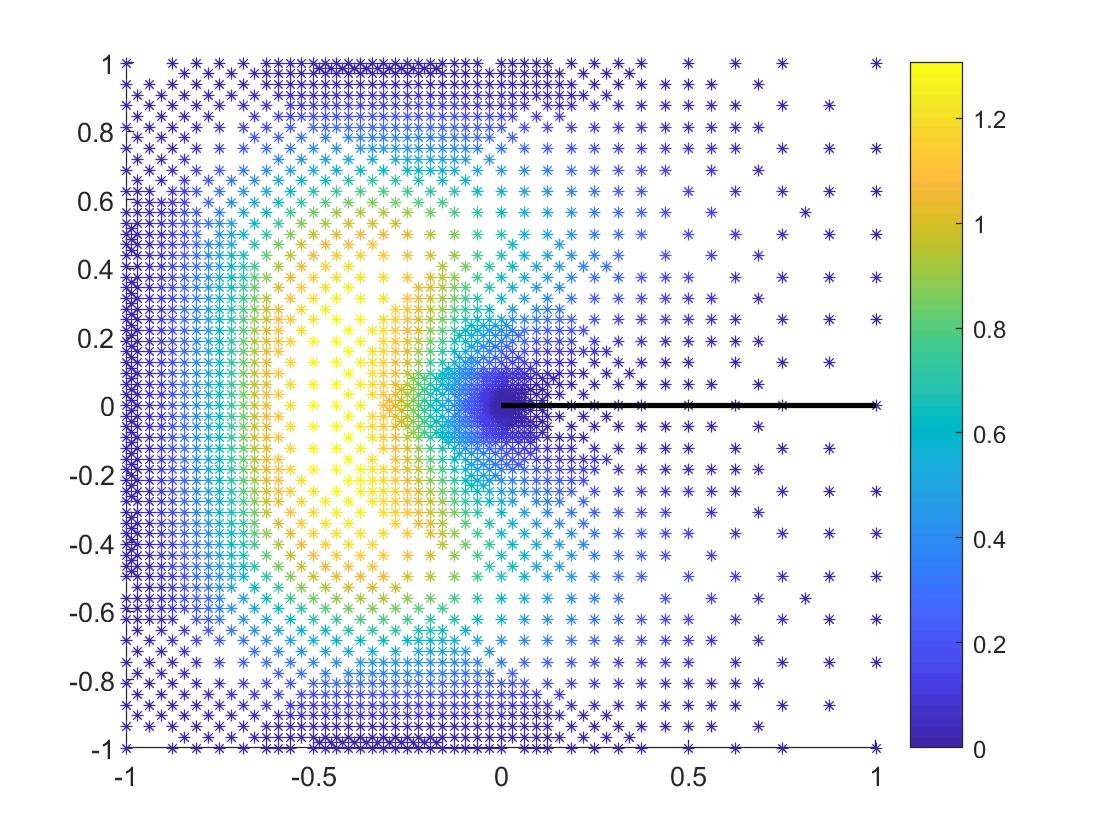}
\includegraphics[width=5cm,height=4cm]{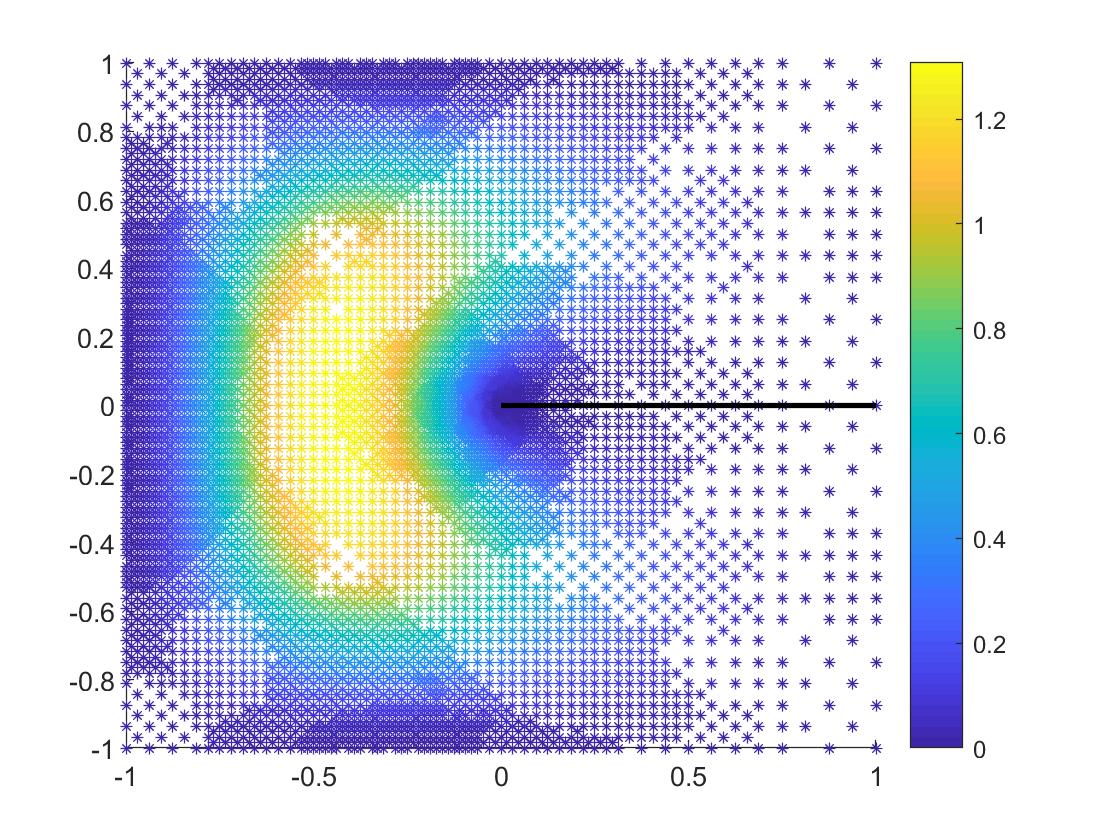}
\includegraphics[width=5cm,height=4cm]{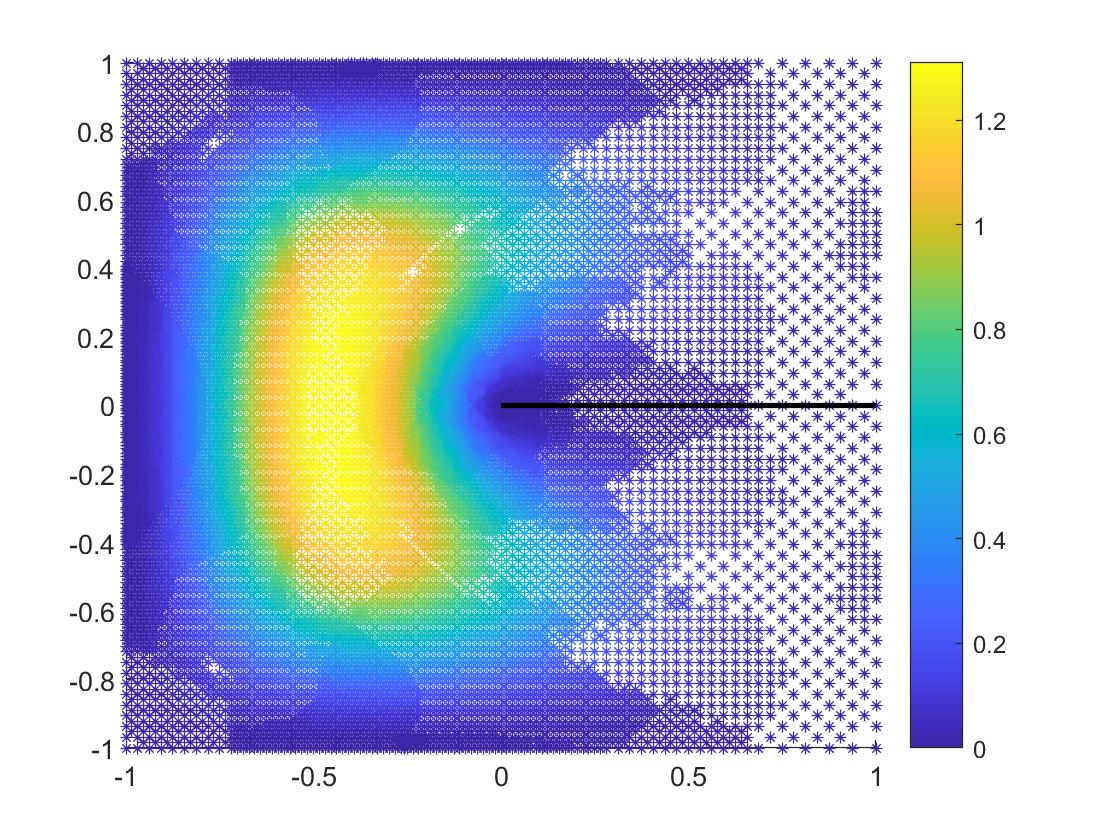} 
\includegraphics[width=5cm,height=4cm]{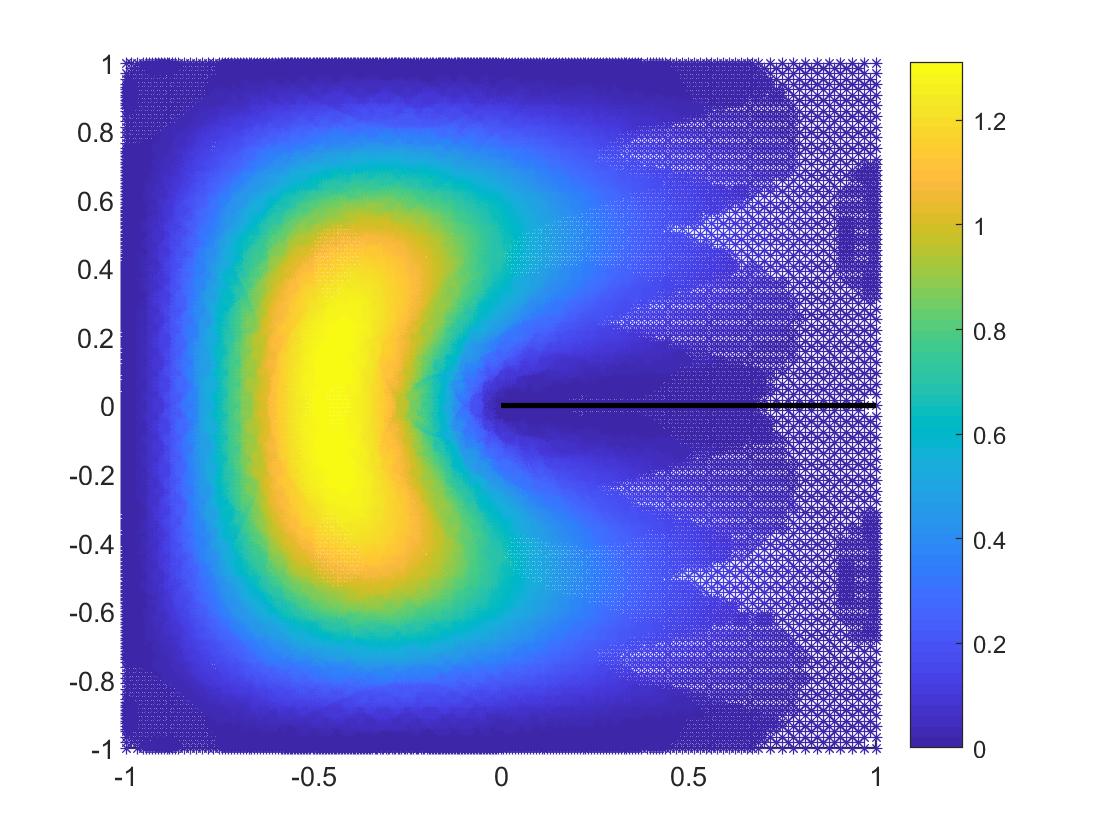}
\caption{\footnotesize{ The first eigenfunctions on adaptive triangulations  for Example 3 with $10246$, $28909$, $76779$ and $459901$ d.o.f, respectively.}}
\label{fig:crack}
\end{figure}

  Consider the model problem \eqref{variance}  on a square domain with a crack 
  $$
  \Omega=(-1,1)^2/[0,1]\times \{0\}.
  $$ 
The boundary condition is $u={\partial u\over \partial n} =0$.
Let $\cT_0$ consist of two right triangles, obtained by cutting the domain $(-1,1)^2$ with a north-east line. Each mesh $\cT_i$ is refined into a half-sized mesh uniformly, to get a higher level mesh $\cT_{i+1}$. Let $\cT_1$ be the initial mesh.

\begin{figure}[!ht]
	\setlength{\abovecaptionskip}{0pt}
	\setlength{\belowcaptionskip}{0pt}
	\centering
	\includegraphics[width=8cm,height=6cm]{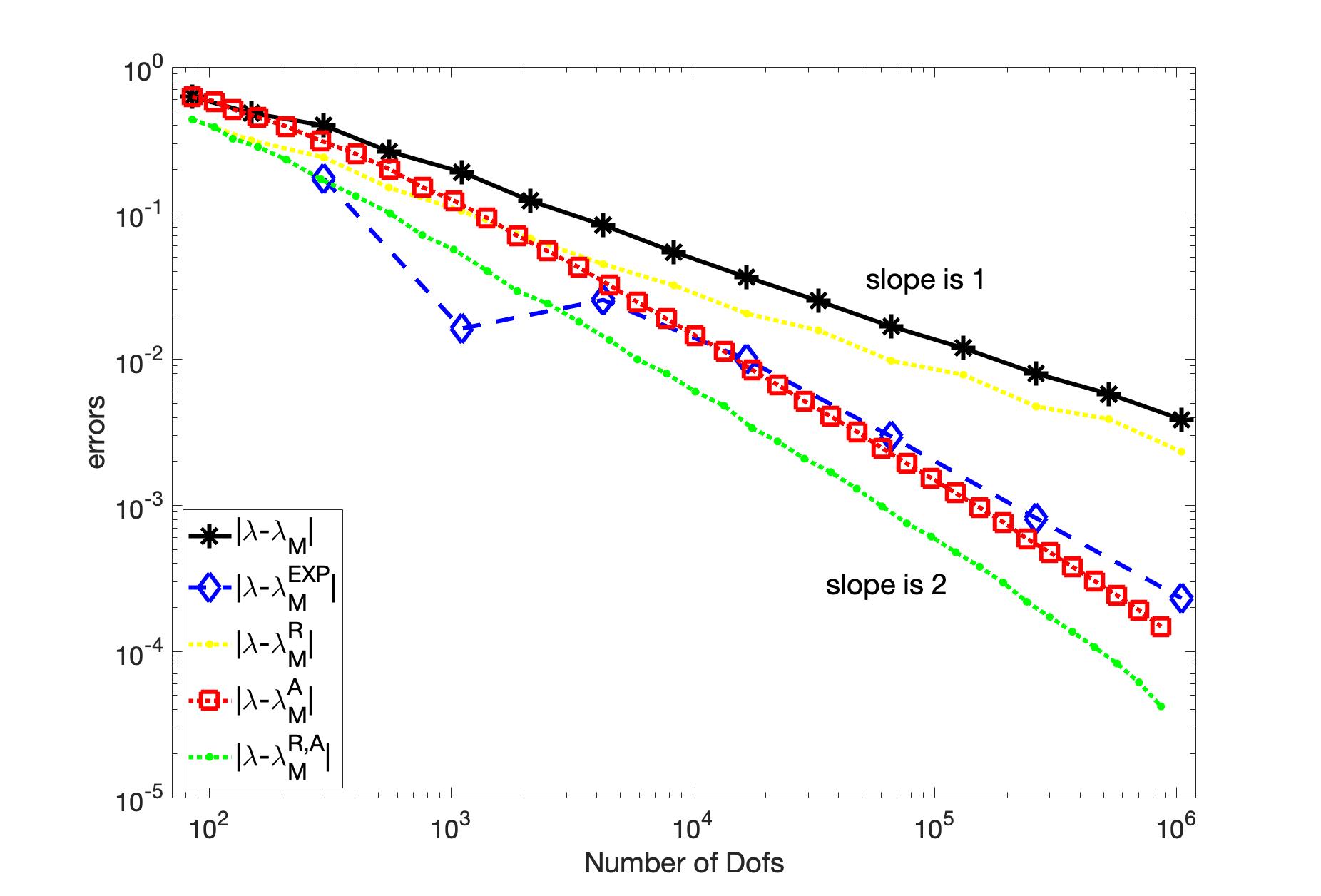}
	\caption{\footnotesize{The relative errors of  the extrapolation eigenvalues on uniform triangulations and the adaptive method for Example~3.}}
	\label{fig:CMorley}
\end{figure}

The eigenfunction with respect to the smallest eigenvalue for this case is singular. Adaptive method is a popular and efficient way to deal with singular cases. 
For eigenvalue problems by the Morley element, an efficient and reliable a posteriori error estimator of Morley elements was proposed  in \cite{shen2015posteriori}. This a posteriori error estimator is adopted here to generate adaptive grids. Denote the smallest approximate eigenvalues by the Morley element on these adaptive grids  by $\lambda_M^{\rm A}$. For the approximate eigenvalues $\lambda_M^{\rm A}$, compute the asymptotically exact a posterior error estimate in \eqref{FdefineCR} and denote the postprocessed approximate eigenvalues in \eqref{lam:F} by $\lambda_M^{\rm R,A}$. Since the eigenvalues of the problem in consideration are unknown, 
we use the adaptive postprocessed eigenvalue $\Mlam^{\rm R, A}$ on an adaptive mesh with 3454396 degrees of freedom as the reference eigenvalue. Fig.~\ref{fig:crack} plots the approximate eigenfunction on adaptive meshes, and
Fig.~\ref{fig:CMorley} plots the errors of approximate eigenvalues $\Mlam$, $\Mlam^{\rm EXP}$, $\Mlam^{\rm R}$, $\Mlam^{\rm A}$ and $\Mlam^{\rm R, A}$. In this case, the discrete eigenvalue $\Mlam$ converges at the rate 1. We compute the extrapolation eigenvalue $\Mlam^{\rm EXP}$ in \eqref{exmethod} with $\alpha=1$. Since the eigenfunction is singular,  the postprocessing scheme improves the accuracy of eigenvalues on uniform meshes without improving the convergence rate, while the extrapolation method can improve the convergence rate to 2.00. Fig.~\ref{fig:CMorley} also implies that the postprocessing scheme is also effective on adaptive meshes.

\subsection{Example 4}
Consider the model problem \eqref{variance}  on a Dumbbell-split domain with a slit $\Omega=(-1,1)\times (-1,5)\backslash ([0,1)\times \{0\}\cup [1,3]\times [-0.75,1])$. 
The boundary condition is $u={\partial u\over \partial n} =0$. The initial triangulation is shown in Fig.~\ref{fig:dumbbell}.
\begin{figure}[!ht]
\begin{center}
\begin{tikzpicture}[xscale=1.7,yscale=1.7]
\draw[ultra thick] (-1,-1) -- (-1,1);
\draw[ultra thick] (-1,1) -- (1,1);
\draw[ultra thick] (1,-0.75) -- (1,1);
\draw[-] (1,-1) -- (1,-0.75);
\draw[ultra thick] (1,-0.75) -- (3,-0.75);
\draw[-] (0.75,-0.75) -- (1,-0.75);
\draw[-] (3,-0.75) -- (3.25,-0.75);
\draw[ultra thick] (3,-0.75) -- (3,1);
\draw[-] (3,-1) -- (3,-0.75);
\draw[ultra thick] (3,1) -- (5,1);
\draw[ultra thick] (5,1) -- (5,-1);
\draw[ultra thick] (-1,-1) -- (5,-1); 
\draw[-] (-1,0) -- (0,1);
\draw[-] (1,0) -- (0,1);
\draw[-] (-1,0) -- (0,-1);
\draw[-] (0,-1) -- (1,0);
\draw[ultra thick] (-1,0) -- (0,0);
\draw[-] (0,0) -- (1,0);
\draw[-] (0,1) -- (0,-1);
\draw[-] (0,0) -- (1,-1);
\draw[-] (0.5,-0.5) -- (1,-0.5);
\draw[-] (1,-0.5) -- (0.5,-1);
\draw[-] (0.5,-0.5) -- (0.5,-1);
\draw[-] (0.75,-0.75) -- (0.75,-1);
\draw[-] (1,-1) -- (1.25,-0.75);
\draw[-] (1.25,-1) -- (1.25,-0.75);
\draw[-] (1.25,-1) -- (1.5,-0.75);
\draw[-] (1.5,-1) -- (1.5,-0.75);
\draw[-] (1.5,-1) -- (1.75,-0.75);
\draw[-] (1.75,-1) -- (1.75,-0.75);
\draw[-] (1.75,-1) -- (2,-0.75);
\draw[-] (2,-1) -- (2,-0.75);
\draw[-] (2.25,-1) -- (2,-0.75);
\draw[-] (2.25,-1) -- (2.25,-0.75);
\draw[-] (2.5,-1) -- (2.25,-0.75);
\draw[-] (2.5,-1) -- (2.5,-0.75);
\draw[-] (2.75,-1) -- (2.5,-0.75);
\draw[-] (2.75,-1) -- (2.75,-0.75);
\draw[-] (3,-1) -- (2.75,-0.75);
\draw[-] (3,0) -- (4,1);
\draw[-] (5,0) -- (4,1);
\draw[-] (5,0) -- (4,-1);
\draw[-] (3,0) -- (4,-1);
\draw[-] (3,0) -- (5,0);
\draw[-] (4,1) -- (4,-1);
\draw[-] (3,-0.5) -- (3.5,-1);
\draw[-] (4,0) -- (3,-1);
\draw[-] (3,-0.5) -- (3.5,-0.5);
\draw[-] (3.5,-1) -- (3.5,-0.5);
\draw[-] (3.25,-1) -- (3.25,-0.75);
\node[below] at (-1,-1) {(-1,-1)};
\node[below] at (1,-1) {(1,-1)};
\node[below] at (3,-1) {(3,-1)};
\node[below] at (5,-1) {(5,-1)};
\node[above] at (-1,1) {(-1,1)};
\node[above] at (1,1) {(1,1)};
\node[above] at (3,1) {(3,1)};
\node[above] at (5,1) {(5,1)};
\node[above,right] at (1,-0.6) {(1,-0.75)};
\node[above,left] at (3,-0.6) {(3,-0.75)}; 
\end{tikzpicture}
\caption{The initial triangulation of Dumbbell domain $\Om$ for Example 4.}
\label{fig:dumbbell}
\end{center}
\end{figure}
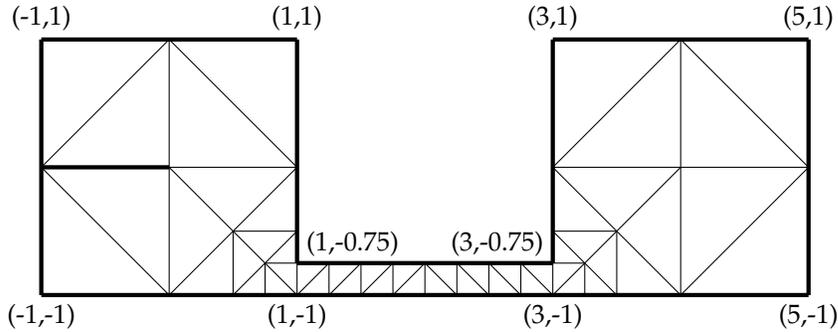

\begin{figure}[!ht]
\setlength{\abovecaptionskip}{0pt}
\setlength{\belowcaptionskip}{0pt}
\centering
\includegraphics[width=13cm,height=6cm]{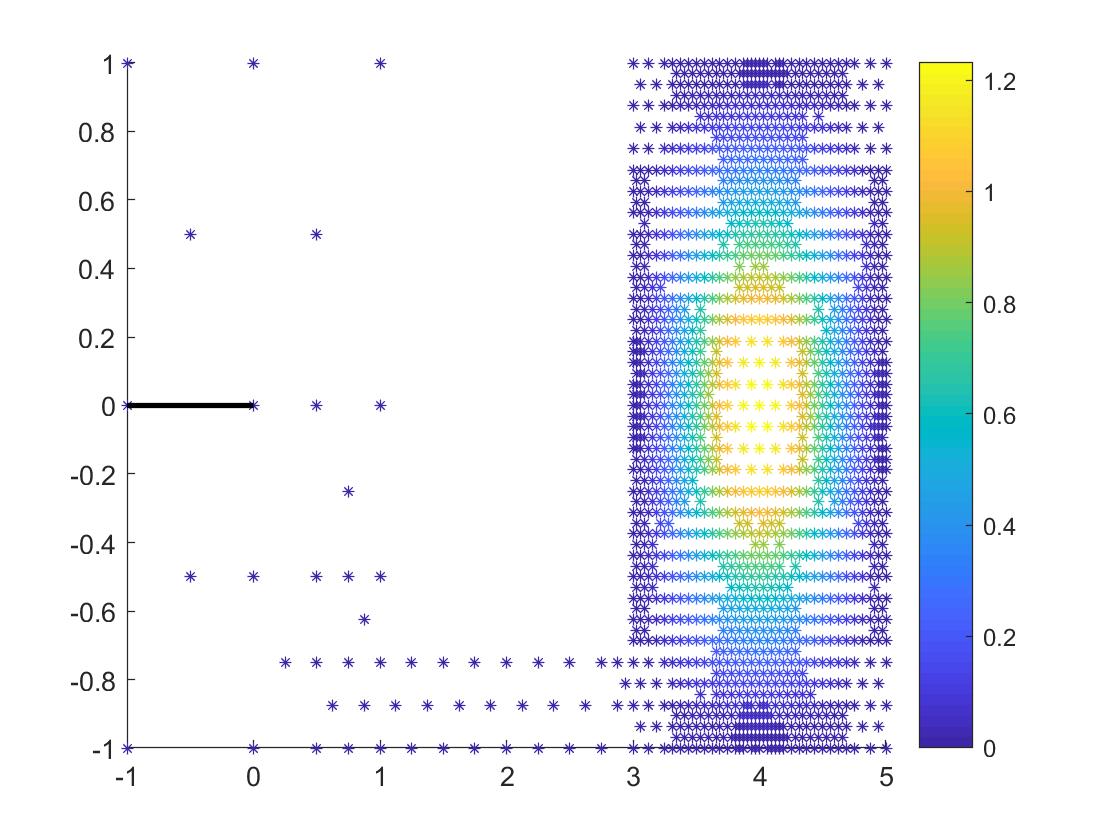} 
\includegraphics[width=13cm,height=6cm]{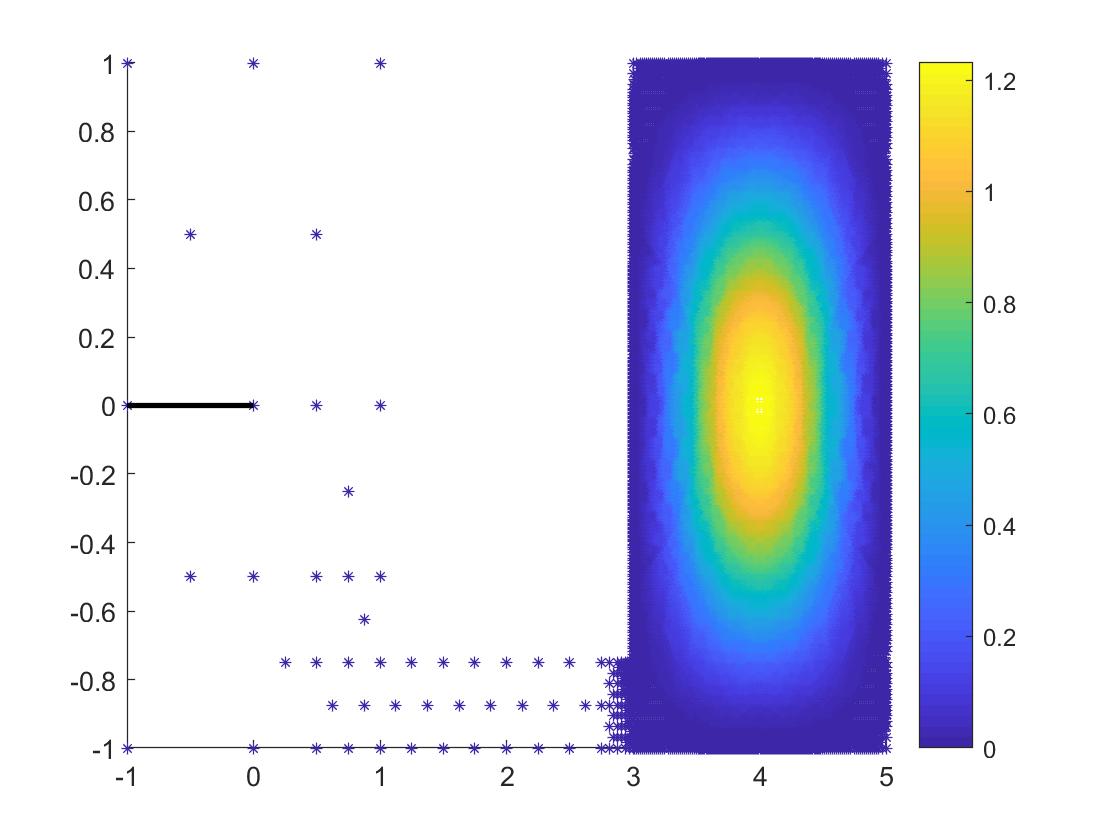}
\caption{\footnotesize{ The first eigenfunctions on adaptive triangulations  for Example 4 with $6409$ and $166591$ d.o.f, respectively.}}
\label{fig:Dumbbell1st}
\end{figure}

Fig.~\ref{fig:Dumbbell1st}  and Fig.~\ref{fig:Dumbbell4th} plot the first and the fourth eigenfunctions on adaptive meshes, respectively.
Fig.~\ref{fig:DMorley} plots the relative errors of the first and the fourth eigenvalues $\Mlam$ by the Morley element, the extrapolation eigenvalue $\Mlam^{\rm EXP}$ in \eqref{exmethod}, the adaptive eigenvalue $\lambda_M^{\rm A}$, and the postprocessed eigenvalues $\lambda_M^{\rm R}$ and $\lambda_M^{\rm R,A}$. 
The first eigenvalue $\Mlam$ converges at the rate~2, while the fourth eigenvalue $\Mlam$ only converges at the rate~1. This implies a relatively higher regularity of the first eigenfunction. This explains why the extrapolation method and postprocessing technique on uniform meshes even have better performance than adaptive methods. For the first eigenvalue, both the extrapolation method and the postprocessing technique can improve the convergence rate to 3.
For the fourth eigenvalue, the convergence rate of eigenvalues by the postprocessing technique stays 
at 1, while that of eigenvalues by the extrapolation method \eqref{exmethod} with $\alpha=1$ increases to~2.  For the case that the eigenfunction is not smooth enough, Fig.~\ref{fig:DMorley} shows that the proposed postprocessing technique is effective on both uniform meshes and adaptive meshes.

\begin{figure}[!ht]
	\setlength{\abovecaptionskip}{0pt}
	\setlength{\belowcaptionskip}{0pt}
	\centering
	\includegraphics[width=6cm,height=5cm]{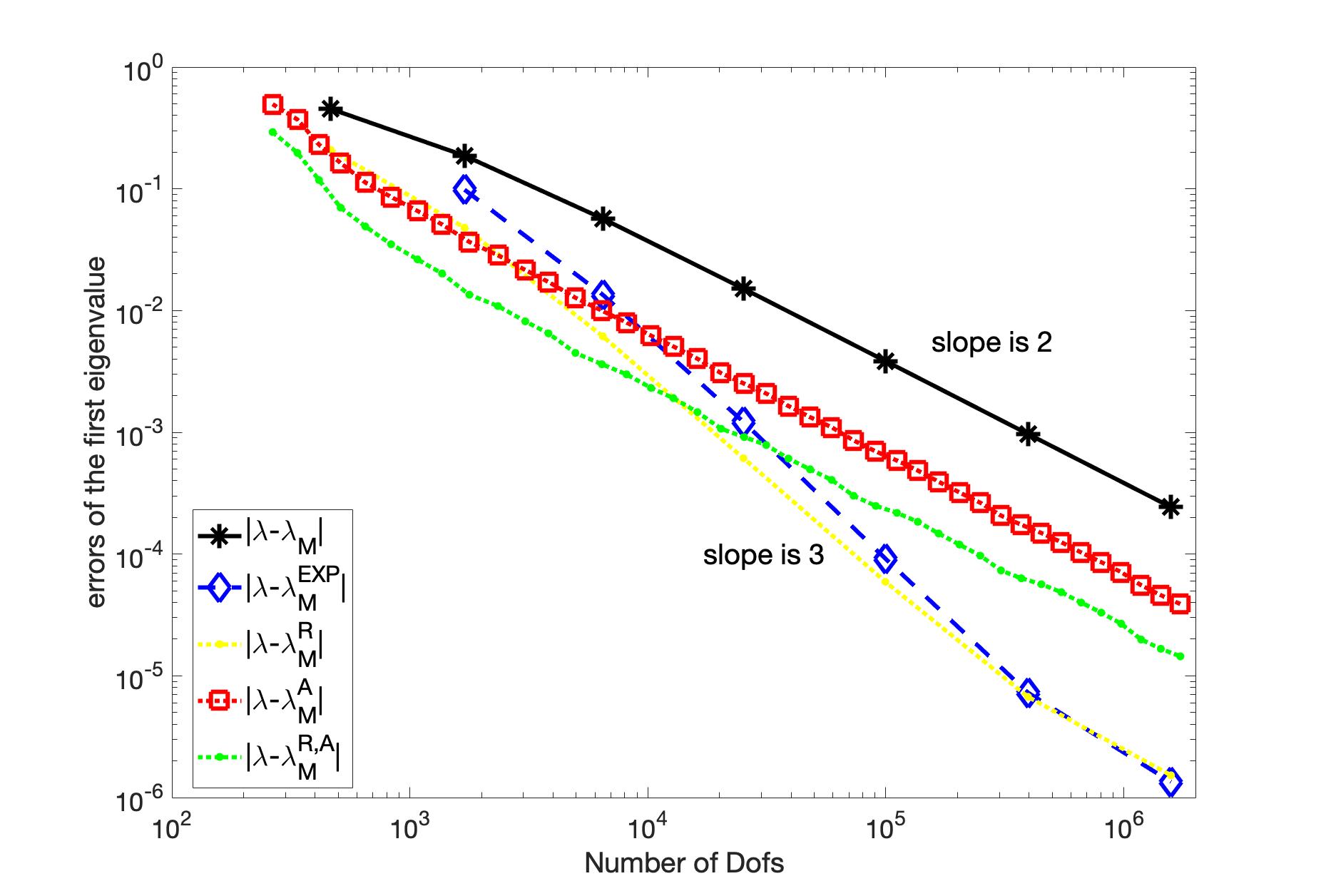}
	\includegraphics[width=6cm,height=5cm]{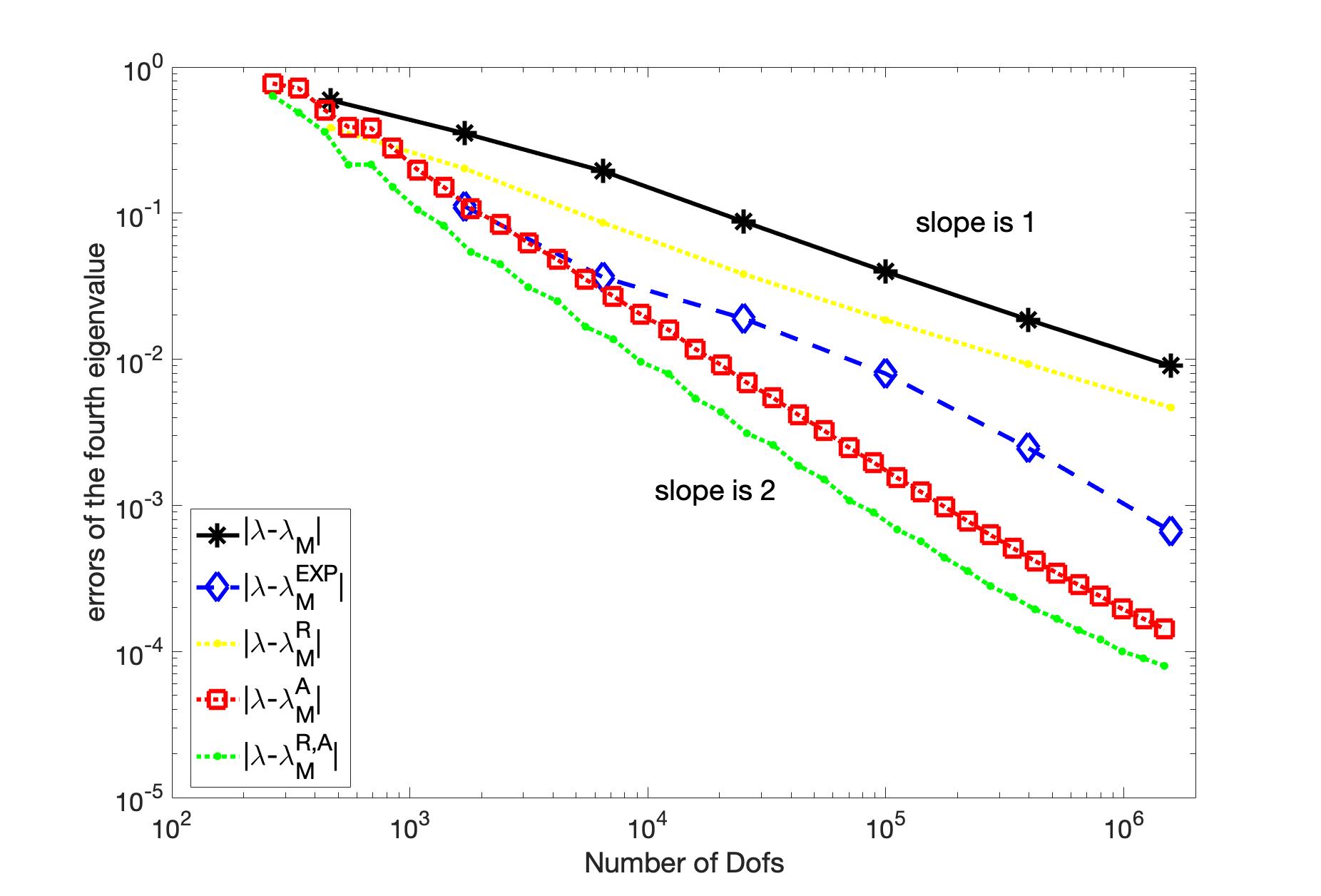} 
	\caption{\footnotesize{The relative errors of the first (left) and the fourth (right) eigenvalue on uniform triangulations and the adaptive method for Example 4.}}
	\label{fig:DMorley}
\end{figure}

\begin{figure}[!ht]
\setlength{\abovecaptionskip}{0pt}
\setlength{\belowcaptionskip}{0pt}
\centering 
\includegraphics[width=13cm,height=6cm]{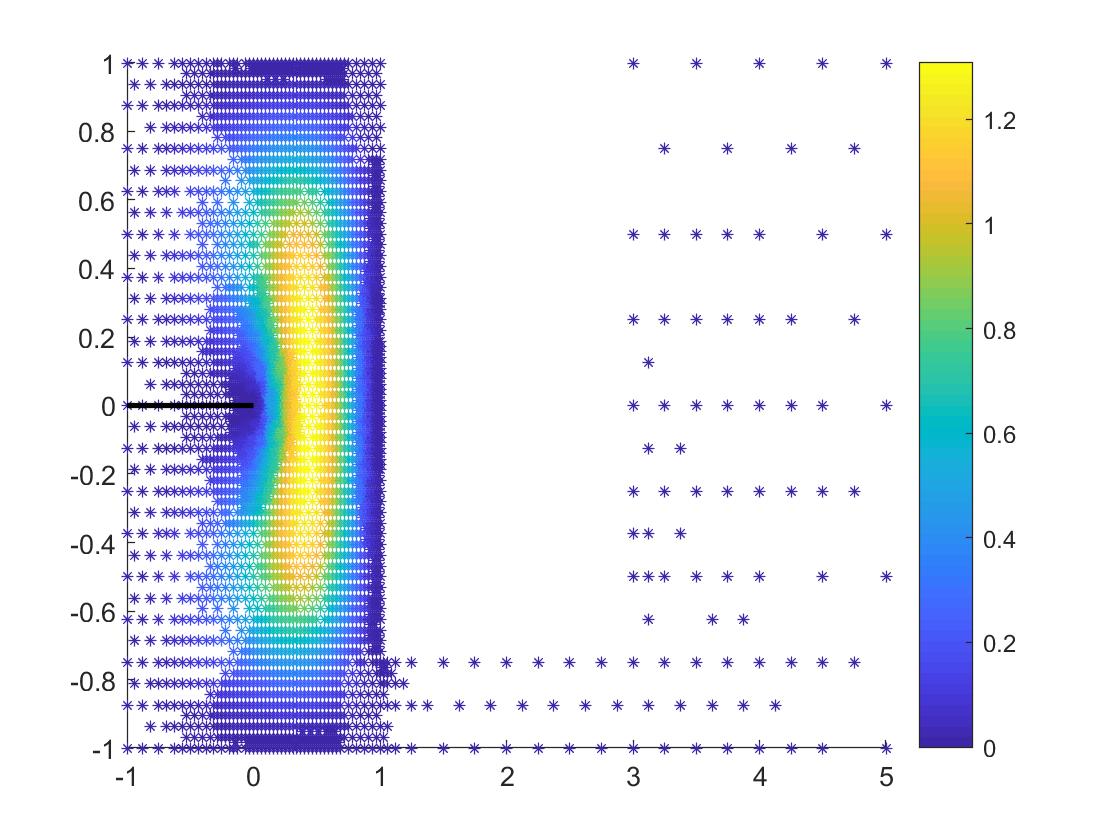} 
\includegraphics[width=13cm,height=6cm]{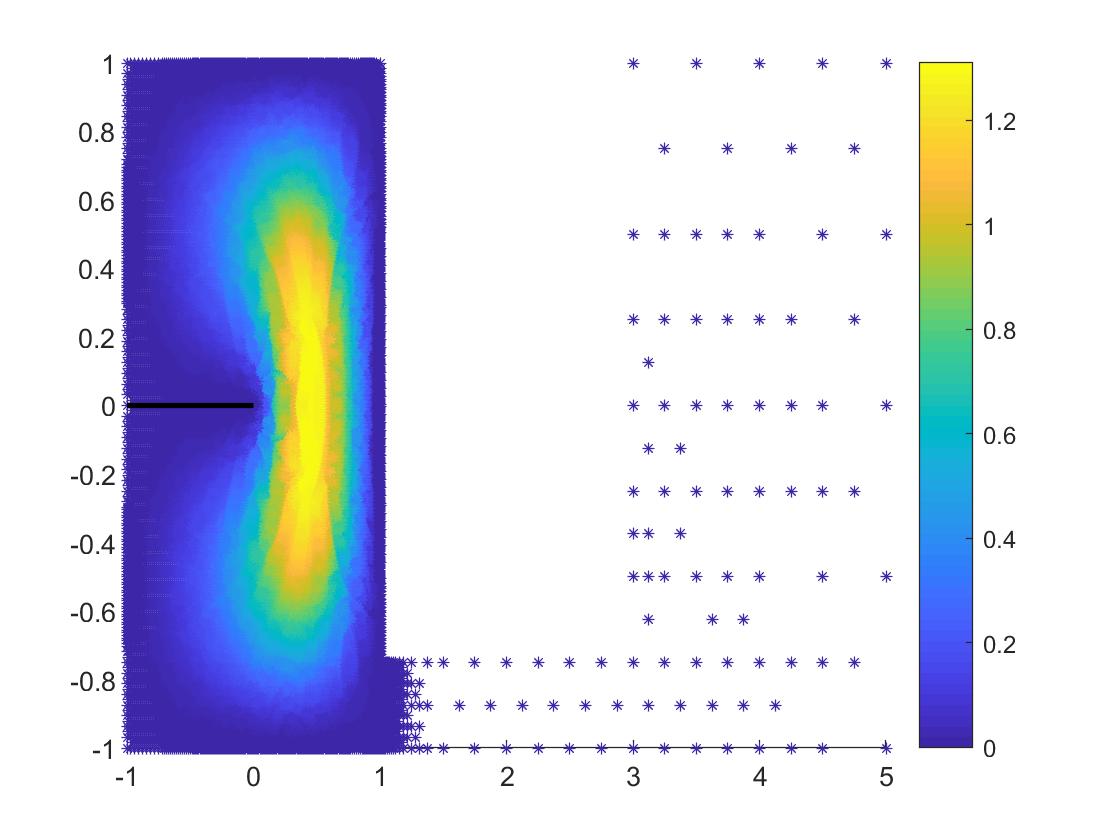}
\caption{\footnotesize{ The fourth eigenfunctions on adaptive triangulations  for Example 4 with $15832$ and $523521$ d.o.f, respectively.}}
\label{fig:Dumbbell4th}
\end{figure}

\section*{Acknowledgments}
The authors are greatly indebted to Professor Jun Hu from Peking University for many useful discussions and the guidance. 
The first author also wishes to  thank the partial support from the Center for Computational Mathematics and Applications, the Pennsylvania State University.

\bibliographystyle{siamplain}
\bibliography{bibifile}

\begin{thebibliography}{10}

\bibitem{arnold1985mixed}
{\sc D.~N. Arnold and F.~Brezzi}, {\em Mixed and nonconforming finite element
  methods: implementation, postprocessing and error estimates}, ESAIM:
  Mathematical Modelling and Numerical Analysis, 19 (1985), pp.~7--32.

\bibitem{arnold2020hellan}
{\sc D.~N. Arnold and S.~W. Walker}, {\em The
  \protect{Hellan--Herrmann--Johnson} method with curved elements}, SIAM
  Journal on Numerical Analysis, 58 (2020), pp.~2829--2855.

\bibitem{babuvska1987estimates}
{\sc I.~Babu{\v{s}}ka and J.~Osborn}, {\em Estimates for the errors in
  eigenvalue and eigenvector approximation by \protect{Galerkin} methods, with
  particular attention to the case of multiple eigenvalues}, SIAM Journal on
  Numerical Analysis, 24 (1987), pp.~1249--1276.

\bibitem{babuvska1989finite}
{\sc I.~Babu{\v{s}}ka and J.~E. Osborn}, {\em Finite element-\protect{Galerkin}
  approximation of the eigenvalues and eigenvectors of selfadjoint problems},
  Mathematics of computation, 52 (1989), pp.~275--297.

\bibitem{Blum1990Finite}
{\sc H.~Blum and R.~Rannacher}, {\em Finite element eigenvalue computation on
  domains with reentrant corners using \protect{Richardson} extrapolation},
  Journal of Computational Mathematics, 8 (1990), pp.~321--332.

\bibitem{boffi2017posteriori}
{\sc D.~Boffi, R.~G. Dur{\'a}n, F.~Gardini, and L.~Gastaldi}, {\em A posteriori
  error analysis for nonconforming approximation of multiple eigenvalues},
  Mathematical Methods in the Applied Sciences, 40 (2017), pp.~350--369.

\bibitem{brandts1994superconvergence}
{\sc J.~H. Brandts}, {\em Superconvergence and a posteriori error estimation
  for triangular mixed finite elements}, Numerische Mathematik, 68 (1994),
  pp.~311--324.

\bibitem{brenner2013morley}
{\sc S.~C. Brenner, L.-y. Sung, H.~Zhang, and Y.~Zhang}, {\em A
  \protect{Morley} finite element method for the displacement obstacle problem
  of clamped \protect{Kirchhoff} plates}, Journal of Computational and Applied
  Mathematics, 254 (2013), pp.~31--42.

\bibitem{carstensen2014guaranteed}
{\sc C.~Carstensen and D.~Gallistl}, {\em Guaranteed lower eigenvalue bounds
  for the biharmonic equation}, Numerische Mathematik, 126 (2014), pp.~33--51.

\bibitem{chen2018multigrid}
{\sc L.~Chen, J.~Hu, and X.~Huang}, {\em Multigrid methods for
  \protect{Hellan--Herrmann--Johnson mixed method of Kirchhoff} plate bending
  problems}, Journal of Scientific Computing, 76 (2018), pp.~673--696.

\bibitem{Chen2007Asymptotic}
{\sc W.~Chen and Q.~Lin}, {\em Asymptotic expansion and extrapolation for the
  eigenvalue approximation of the biharmonic eigenvalue problem by
  \protect{Ciarlet-Raviart} scheme}, Advances in Computational Mathematics, 27
  (2007), pp.~95--106.

\bibitem{ciarlet1974conforming}
{\sc P.~G. Ciarlet}, {\em Conforming and nonconforming finite element methods
  for solving the plate problem}, in Conference on the Numerical Solution of
  Differential Equations, Springer, 1974, pp.~21--31.

\bibitem{comodi1989hellan}
{\sc M.~Comodi}, {\em The \protect{Hellan-Herrmann-Johnson} method: some new
  error estimates and postprocessing}, Mathematics of Computation, 52 (1989),
  pp.~17--29.

\bibitem{Crouzeix1973Conforming}
{\sc M.~Crouzeix and P.-A. Raviart}, {\em Conforming and nonconforming finite
  element methods for solving the stationary \protect{Stokes} equations}, Revue
  fran{\c{c}}aise d'automatique informatique recherche op{\'e}rationnelle.
  Math{\'e}matique, 7 (1973), pp.~33--75.

\bibitem{da2007posteriori}
{\sc L.~B. da~Veiga, J.~Niiranen, and R.~Stenberg}, {\em A posteriori error
  estimates for the \protect{Morley} plate bending element}, Numerische
  Mathematik, 106 (2007), pp.~165--179.

\bibitem{Ding1990quadrature}
{\sc Y.~Ding and Q.~Lin}, {\em Quadrature and extrapolation for the variable
  coefficient elliptic eigenvalue problem}, Systems Science Mathematical
  Sciences, 3 (1990), pp.~327--336.

\bibitem{gallistl2015morley}
{\sc D.~Gallistl}, {\em Morley finite element method for the eigenvalues of the
  biharmonic operator}, IMA Journal of Numerical Analysis, 35 (2015),
  pp.~1779--1811.

\bibitem{Hu2014Lower}
{\sc J.~Hu, Y.~Huang, and Q.~Lin}, {\em Lower bounds for eigenvalues of
  elliptic operators: By nonconforming finite element methods}, Journal of
  Scientific Computing, 61 (2014), pp.~196--221.

\bibitem{hu2016guaranteed}
{\sc J.~Hu, Y.~Huang, and R.~Ma}, {\em Guaranteed lower bounds for eigenvalues
  of elliptic operators}, Journal of Scientific Computing, 67 (2016),
  pp.~1181--1197.

\bibitem{hu2020asymptotically}
{\sc J.~Hu and L.~Ma}, {\em Asymptotically exact a posteriori error estimates
  of eigenvalues by the \protect{Crouzeix--Raviart element and enriched
  Crouzeix--Raviart} element}, SIAM Journal on Scientific Computing, 42 (2020),
  pp.~A797--A821.

\bibitem{hu2019asymptotic}
{\sc J.~Hu and L.~Ma}, {\em Asymptotic expansions of eigenvalues by both the
  \protect{Crouzeix-Raviart and enriched Crouzeix-Raviart} elements},
  Mathematics of Computation, accepted,  (2021).

\bibitem{hu2021optimal}
{\sc J.~Hu, L.~Ma, and R.~Ma}, {\em Optimal superconvergence analysis for the
  \protect{Crouzeix-Raviart} and the \protect{Morley} elements}, Advances in
  Computational Mathematics, 47 (2021), pp.~1--25.

\bibitem{Hu2015The}
{\sc J.~Hu and R.~Ma}, {\em The \protect{enriched Crouzeix-Raviart} elements
  are equivalent to the \protect{Raviart-Thomas} elements}, Journal of
  Scientific Computing, 63 (2015), pp.~410--425.

\bibitem{Hu2016Superconvergence}
{\sc J.~Hu and R.~Ma}, {\em Superconvergence of both the
  \protect{Crouzeix-Raviart} and \protect{Morley} elements}, Numerische
  Mathematik, 132 (2016), pp.~491--509.

\bibitem{hu2009new}
{\sc J.~Hu and Z.~Shi}, {\em A new a posteriori error estimate for the
  \protect{Morley} element}, Numerische Mathematik, 112 (2009), pp.~25--40.

\bibitem{hu2012lower}
{\sc J.~Hu and Z.~Shi}, {\em The lower bound of the error estimate in the
  \protect{$L^2$ norm for the Adini} element of the biharmonic equation}, arXiv
  preprint arXiv:1211.4677,  (2012).

\bibitem{hu2012convergence}
{\sc J.~Hu, Z.~Shi, and J.~Xu}, {\em Convergence and optimality of the adaptive
  \protect{Morley} element method}, Numerische Mathematik, 121 (2012),
  pp.~731--752.

\bibitem{huang2008superconvergence}
{\sc Y.~Huang and J.~Xu}, {\em Superconvergence of quadratic finite elements on
  mildly structured grids}, Mathematics of computation, 77 (2008),
  pp.~1253--1268.

\bibitem{Jia2010Approximation}
{\sc S.~Jia, H.~Xie, X.~Yin, and S.~Gao}, {\em Approximation and eigenvalue
  extrapolation of biharmonic eigenvalue problem by nonconforming finite
  element methods}, Numerical Methods for Partial Differential Equations, 24
  (2010), pp.~435--448.

\bibitem{johnson1973convergence}
{\sc C.~Johnson}, {\em On the convergence of a mixed finite-element method for
  plate bending problems}, Numerische Mathematik, 21 (1973), pp.~43--62.

\bibitem{li2014new}
{\sc M.~Li, X.~Guan, and S.~Mao}, {\em New error estimates of the
  \protect{Morley} element for the plate bending problems}, Journal of
  Computational and Applied Mathematics, 263 (2014), pp.~405--416.

\bibitem{Lin2008New}
{\sc Q.~Lin, H.-T. Huang, and Z.-C. Li}, {\em New expansions of numerical
  eigenvalues for $-\bigtriangleup u=\lambda \rho u$ by nonconforming
  elements}, Mathematics of Computation, 77 (2008), pp.~2061--2084.

\bibitem{Lin2009New}
{\sc Q.~Lin, H.-T. Huang, and Z.-C. Li}, {\em New expansions of numerical
  eigenvalues by \protect{Wilson's} element}, Journal of Computational and
  Applied Mathematics, 225 (2009), pp.~213--226.

\bibitem{Lin2007Finite}
{\sc Q.~Lin and J.~Lin}, {\em Finite element methods: Accuracy and
  \protect{Improvement}}, China Sci. Press, Beijing,  (2006).

\bibitem{Lin1984Asymptotic}
{\sc Q.~Lin and T.~Lu}, {\em Asymptotic expansions for finite element
  eigenvalues and finite element}, Bonn. Math. Schrift, 158 (1984), pp.~1--10.

\bibitem{lin1999high}
{\sc Q.~Lin and D.~Wu}, {\em High-accuracy approximations for eigenvalue
  problems by the \protect{Carey} non-conforming finite element}, International
  Journal for Numerical Methods in Biomedical Engineering, 15 (1999),
  pp.~19--31.

\bibitem{Lin2009Asymptotic}
{\sc Q.~Lin and H.~Xie}, {\em Asymptotic error expansion and richardson
  extrapolation of eigenvalue approximations for second order elliptic problems
  by the mixed finite element method}, Applied numerical mathematics, 59
  (2009), pp.~1884--1893.

\bibitem{lin2010new}
{\sc Q.~Lin and H.~Xie}, {\em New expansions of numerical eigenvalue for
  \protect{$- \Delta u= \lambda \rho u$} by linear elements on different
  triangular meshes}, Internat Journal of Information Systems Sciences, 6
  (2010), pp.~10--34.

\bibitem{Lin2011Extrapolation}
{\sc Q.~Lin, J.~Zhou, and H.~Chen}, {\em Extrapolation of three-dimensional
  eigenvalue finite element approximation}, Mathematics in Practice Theory, 11
  (2011), pp.~132--139.

\bibitem{morley1968triangular}
{\sc L.~S.~D. Morley}, {\em The triangular equilibrium element in the solution
  of plate bending problems}, The Aeronautical Quarterly, 19 (1968),
  pp.~149--169.

\bibitem{naga2006enhancing}
{\sc A.~Naga, Z.~Zhang, and A.~Zhou}, {\em Enhancing eigenvalue approximation
  by gradient recovery}, SIAM Journal on Scientific Computing, 28 (2006),
  pp.~1289--1300.

\bibitem{rannacher1979nonconforming}
{\sc R.~Rannacher}, {\em Nonconforming finite element methods for eigenvalue
  problems in linear plate theory}, Numerische Mathematik, 33 (1979),
  pp.~23--42.

\bibitem{shen2015posteriori}
{\sc Q.~Shen}, {\em A posteriori error estimates of the \protect{Morley}
  element for the fourth order elliptic eigenvalue problem}, Numerical
  Algorithms, 68 (2015), pp.~455--466.

\bibitem{ming2006morley}
{\sc M.~Wang and J.~Xu}, {\em The \protect{Morley} element for fourth order
  elliptic equations in any dimensions}, Numerische Mathematik, 103 (2006),
  pp.~155--169.

\end{thebibliography}

\end{document}